\documentclass{amsart}

\usepackage[hiresbb]{graphicx}
\usepackage{amsmath,amssymb}
\usepackage{amsthm}
\usepackage{amsfonts}
\usepackage{amscd}
\newtheorem{df}{Definition}[section]
\newtheorem{thm}[df]{Theorem}
\newtheorem{prop}[df]{Proposition}
\newtheorem{lemm}[df]{Lemma}
\newtheorem{cor}[df]{Corollary}

\newtheorem{rem}[df]{Remark}

\newcommand{\id}{\mathrm{id}}
\newcommand{\Q}{\mathbb{Q}}

\newcommand{\Z}{\mathbb{Z}}

\newcommand{\shuugou}[1]{\{ #1 \}}
\newcommand{\zettaiti}[1]{\lvert #1 \rvert}

\newcommand{\im}{\mathrm{im}}

\newcommand{\gyaku}[1]{ #1^{-1}}

\newcommand{\skein}[1]{\mathcal{S}( #1 )}

\newcommand{\kukakko}[1]{\langle #1 \rangle}
\newcommand{\defeq}{\stackrel{\mathrm{def.}}{=}}
\newcommand{\Aut}{\mathrm{Aut}}
\newcommand{\bch}{\mathrm{bch}}

\newcommand{\filtn}[1]{\{ #1 \}_{n \geq 0}}

\newcommand{\comp}[1]{\underleftarrow{\lim}_{#1 \rightarrow \infty}}
\newcommand{\gauss}[1]{\lfloor #1 \rfloor}

\newcommand{\ad}{\mathrm{ad}}
\newcommand{\arccosh }{\mathrm{arccosh}}

\begin{document}

\title[The quotients of Kauffman bracket skein algebras ]
{The quotient of a Kauffman bracket skein algebra by
the square of an augmentation ideal}
\author{Shunsuke Tsuji}
\address{Graduate School of Mthematicla Sciences, The University of Tokyo, 
3-8-1 Komaba, Meguro-ku, Tokyo 153-8941, Japan}
\email{tsujish@ms.u-tokyo.ac.jp}
\date{}
\maketitle

\begin{abstract}
We give an explicit basis $\mathcal{B}$ of
the quotient of the Kauffman bracket  skein algebra $\mathcal{S} (\Sigma)$
on a surface $\Sigma$ by the square of an augmentation ideal.
As an application, it induces two kinds of finite type invariants 
of links in a handle body
in the sense of  Le \cite{Le1996}. Moreover, we construct an 
embedding of the mapping class group of a compact connected surface
of genus $0$ into the Kauffman bracket skein algebra on the surface
completed with respect to a filtration coming from
the augmentation ideal.
\end{abstract}

\section{Introduction}
The discovery of the Jones polynomial made a remarkable
development in knot theory. L. Kauffman gave an effective
method to compute the Jones polynomial by the Kauffman bracket,
which is an invariant of a framed unoriented link 
assigning a Laurent polynomial in a variable $A$.
The Kauffman bracket is computed by some skein relation.
The Kauffman bracket module in a handle body defined
by this skein relation is useful in the study of knots and links in
the handle body.
For example, see  Lickorish \cite{skeins_and_handlebodies},
Prztycki \cite{skeinmodule} and our paper \cite{TsujiCSAI}
section 5. Furthermore, recently, we find a new relationship
between the study of the mapping class group of a surface
and the Kauffman bracket skein modules in the product of 
the surface and the closed interval $[0,1]$.
In particular, we obtain a formula  
of the action of an Dehn twist $t_c$
on the completed skein module $\widehat{\skein{\Sigma,J}}$
in terms of the inverse
function of hyperbolic cosine function

\begin{equation}
\label{equation_skein_Dehn}
\exp(\sigma(\frac{-A+\gyaku{A}}{4\log(-A)} (\arccosh (-\frac{c}{2}))^2))(\cdot)
=t_c (\cdot): \widehat{\skein{\Sigma,J}} \to \widehat{\skein{\Sigma,J}}.
\end{equation}
\cite{TsujiCSAI} 
Theorem 4.5.
This formula is an analogy of a formula for the action of
$t_c$ on the completed group ring of the fundamental group of the surface
\cite{Kawazumi} \cite{KK} \cite{MT}.

To describe the relationship more precisely,
we need to clarify the Kauffman bracket skein module.
In this paper, we construct an embedding of the mapping class group
of a compact connected oriented 
surface of genus $0$ into the completed Kauffman bracket skein algebra
on the surface in Theorem \ref{thm_zeta_pure_braid}.
In our succeeding paper \cite{TsujiTorelli},
 we also construct an embedding of the Torelli group
of a compact connected oriented surface with non-empty connected boundary
into the completed Kauffman bracket skein algebra on the surface.

In \cite{TsujiCSAI}, we introduce a filtration of the skein module
on a surface
and its completion,
in order to define the logarithm of a Dehn  twist,
 the exponential of the action of the skein algebra on
the surface and the formula (\ref{equation_skein_Dehn}).
In this paper, we introduce another filtration of 
the skein algebra in order to
consider the logarithm of
other elements of the mapping class group
of  the surface.
In particular, using this new filtration,
we can define the logarithm of any element of 
the mapping class group of a surface of genus $0$.
See Theorem \ref{thm_zeta_pure_braid}.
In our subsequent papers, we consider the logarithm of
any element of the Torelli group of a surface
with non-empty connected boundary.
In order to define the filtration,
we give an explicit basis of the quotient 
of the Kauffman bracket skein algebra 
by the square of an augmentation ideal.

Let $\Sigma$ be a compact connected oriented surface, $I$
the closed interval $[0,1]$, $\skein{\Sigma}$
the Kauffman bracket skein algebra on $\Sigma$ and
$\epsilon$ the augmentation map defined by 
$\epsilon (A+1) =0 $ and $\epsilon ([L]-(-2)^{\zettaiti{L}}) =0$.
Here $\zettaiti{L}$ is the number of the components of $L$.
For a link $L$ in $\Sigma \times I$, the element
$(-2)^{-\zettaiti{L}}[L] \in \skein{\Sigma} / (\ker \epsilon)^n$ is a
finite invariant of  order  $n$ in the sense of Le \cite{Le1996} (3.2).
By \cite{TsujiCSAI} Lemma 5.3, we actually have
\begin{equation*}
\sum_{L' \subset L} (-1)^{\zettaiti{L'}}(-2)^{-\zettaiti{L'}}[L'] \in (\ker \epsilon)^n
\end{equation*}
 for a link $L$ having components more than $n$,
where the sum is over all sublinks $L' \subset L$
including the empty link.

Now we assume $\Sigma$ has a non-empty boundary.
Then we introduce a family $\mathcal{B}$ 
of elements of $\skein{\Sigma}/(\ker \epsilon)^2$ by
\begin{align*}
\mathcal{B} \defeq
\shuugou{1} \cup \shuugou{A+1} \cup \shuugou{\kukakko{x_i,x_j}|i \leq j} \cup
\shuugou{\kukakko{x_i,x_j,x_k}|i < j < k}
\end{align*}
where the fundamental group $\pi_1 (\Sigma)$ is freely generated by
$x_1, \cdots, x_M$.
For details, see Lemma \ref{lemm_generator}.
In section \ref{section_computation}, we prove
the set $\mathcal{B}$ generates $\skein{\Sigma}/(\ker \epsilon)^2$ as
$\Q$ vector space.
In section \ref{section_basis}, we prove the $\Q$-linear independence
of the set $\mathcal{B}$.
To prove the independence, we define a bilinear form of 
$\vartheta ((\cdot)(\cdot)) :
\skein{\Sigma} \times \skein{\Sigma} \to \Q[A^{\pm1}]$.
This bilinear form is non-degenerate, in other words,
for any $x \in \skein{\Sigma} \backslash \shuugou{0}$, there exists
$y \in \skein{\Sigma} $ satisfying $\vartheta(xy) \neq 0$.
Furthermore, we have $\vartheta( (\ker \epsilon)^n (\ker \epsilon)^m)
\in  ((A+1)^{n+m})$.
Using this basis, we obtain an explicit map
\begin{equation*}
\Q \mathcal{T} (\Sigma) \to \skein{\Sigma}/(\ker \epsilon)^2 \simeq \Q \mathcal{B},
L \mapsto (-2)^{-\zettaiti{L}} [L]
\end{equation*}
which is a finite type invariant of order $2$
for links in $\Sigma \times I$,
where $\mathcal{T}(\Sigma)$ is the set of 
unoriented framed link in $\Sigma \times I$.

As an application of the basis, we introduce a new filtration
$\filtn{F^n \skein{\Sigma}}$ satisfying $F^{2n} \skein{\Sigma} = (\ker \epsilon)^n$,
 $F^3 \skein{\Sigma}/F^4 \skein{\Sigma} =\Q \shuugou{\kukakko{x_i,x_j,x_k}|
i<j<k}$ and $F^{2n+1} \skein{\Sigma} =F^3 \skein{\Sigma} F^{2n-2} \skein{\Sigma}$.
We remark that
\begin{align*}
&F^2 \skein{\Sigma}/F^3 \skein{\Sigma} =\Q \oplus S^2 (H_1 (\Sigma,\Q)), \\
&F^3 \skein{\Sigma}/F^4 \skein{\Sigma} =\wedge^3 (H_1(\Sigma,\Q)),
\end{align*}
where we denote the second symmetric tensor of $H_1(\Sigma,\Q)$
by $S^2 (H_1 (\Sigma,\Q))$ and the third exterior power of 
$H_1 (\Sigma,\Q)$ by $\wedge^3 (H_1(\Sigma,\Q))$.
This filtration is finer than $\filtn{(\ker \epsilon)^n}$.
In fact, by Proposition \ref{prop_product_filtration} and
Proposition \ref{prop_bracket_filtration}, we have
\begin{align*}
&F^{2n+1}\skein{\Sigma}F^{2m+1} \skein{\Sigma} \subset F^{2n+2m+2} \skein{\Sigma}, \\
&[F^{2n+1} \skein{\Sigma}, F^{2m+1} \skein{\Sigma}] 
\subset F^{2n+2m} \skein{\Sigma},
\end{align*}
where the bracket $[ \ \ , \ \ ]$ is defined by $[x,y] \defeq \frac{1}{-A+\gyaku{A}}
(xy-yx)$.
Furthermore, by Corollary \ref{cor_filtration_bilinear}, we have
\begin{equation}
\label{equation_inner_product}
\vartheta (F^{n} \skein{\Sigma} F^{m} \skein{\Sigma})
\subset ((A+1)^{\gauss{\frac{n+m+1}{2}}}).
\end{equation}
where $\gauss{x}$ is the largest integer not greater than $x$ for $x \in \Q$.
By the equation, $\vartheta$ induces
$\vartheta_{n} : F^n \skein{\Sigma}/F^{n+1} \skein{\Sigma}
\times F^n \skein{\Sigma}/F^{n+1} \skein{\Sigma} \to
((A+1)^{n})/((A+1)^{n+1}) \simeq \Q$.
By the proof of the independence of $\skein{\Sigma}/(\ker \epsilon)^2$,
$\vartheta_2$ and $\vartheta_3$ are non-degenerate.
Furthermore, for $n=2,3$ and $x \in
F^n \skein{\Sigma} /F^{n+1} \skein{\Sigma}$,
 we have $\vartheta_n (x,x) =0$ if and only if $x =0$.

We define an evaluation map $\mathrm{ev}:\Q \mathcal{T} (\Sigma) \to
\mathrm{Hom}_{\Q} (F^{2n-1} \skein{\Sigma}, \Q[A^\pm])$
by 
\begin{equation*}
(\mathrm{ev}(L))(x) \defeq \frac{1}{(A+1)^n} \vartheta ((-2)^{-\zettaiti{L}}[L],
x).
\end{equation*}
Using the equation (\ref{equation_inner_product}),
the $\Q$-linear map $\mathrm{ev}$
induces a finite type invariant
$\mathrm{ev}:\Q \mathcal{T} (\Sigma) \to
\mathrm{Hom}_{\Q} (F^{2n-1} \skein{\Sigma}, \Q[A^\pm]/
(A+1)^m)$ of order $m$.

Let $\Sigma$ be a compact connected oriented  of genus $0$.
By Lemma \ref{lemm_filt_pure}, we can consider the 
Baker-Campbell-Hausdorff series on
$\widehat{\skein{\Sigma}}$.
In section \ref{section_pure_braid_group},
as an application of this filtration,
we construct an embedding $\zeta$ of the mapping class group $\mathcal{M} (\Sigma)$
of $\Sigma$ into $\widehat{\skein{\Sigma}}$
with respect to a group law using the 
Baker-Campbell-Hausdorff series.
Furthermore we have
\begin{equation*}
\xi (\cdot) =\exp (\sigma(\zeta(\xi)))(\cdot):\widehat{\skein{\Sigma,J}} \to \widehat{\skein
{\Sigma,J}}
\end{equation*}
for any $\xi \in \mathcal{M} (\Sigma)$ and $J$ the finite subset of $\partial \Sigma$.

\section*{Acknowledgment}
The author would like to thank his adviser, Nariya Kawazumi, for helpful discussion
and encouragement. He is also grateful to Kazuo Habiro, Gw\'{e}na\"{e}l Massuyeau, Jun Murakami, Tomotada Ohtsuki and Masatoshi Sato for helpful comments.
In particular, the author  thanks to Jun Murakami
for valuable comments about the theory of 3-dimensinal manifolds, knots and links.
This work was supported by JSPS KAKENHI Grant Number 15J05288 
and the Leading Graduate Course for Frontiers of Mathematical Sciences and Phsyics.

\tableofcontents
\section{Definition and Review}
\label{section_definition_review}

We review some definitions and facts about the Kauffman bracket skein algebra 
on a surface, for details, see \cite{TsujiCSAI}.

Through this section, let $\Sigma$ be a compact connected surface and $I$ the closed interval $[0,1]$.

\subsection{Kauffman bracket skein algebras and modules}

Let $J$ be a finite subset of $\partial \Sigma$.
We denote by $\mathcal{T}(\Sigma,J)$ the set of unoriented framed tangles in
$\Sigma \times I$ with base point set $J$ \cite{TsujiCSAI} Definition 2.1. 
For a tangle diagram $d$, we denote by $T(d)$ the tangle presented by $d$.
The Kauffman bracket skein module $\skein{\Sigma,J}$
is the quotient of $\Q [A.\gyaku{A}] \mathcal{T}(\Sigma,J)$ by the skein relation and the
trivial knot relation \cite{TsujiCSAI} Definition 3.2. 
The skein relation is 
\begin{equation*}
T(d_1) =AT(d_\infty)+\gyaku{A} T(d_0)
\end{equation*}
where $d_1$, $d_\infty$ and $d_0$ are differ
only in an open disk shown in Figure 
\ref{fig_K2},
Figure \ref{fig_Kinfi} and Figure \ref{fig_K0}, respectively.
The trivial knot relation is
\begin{equation*}
T(d) =(-A^2-A^{-2})T(d')
\end{equation*}
where $d$ and $d'$ are differ only in a open disk shown 
a boundary of a disk and empty, respectively.
We remark that  we don't assume  "the boundary skein relation"
and "the value of  a contractible arc" in Muller \cite{Mu2012}.
We write simply $\skein{\Sigma} \defeq \skein{\Sigma, \emptyset}$.
The element of $\skein{\Sigma,J}$ represented by $T \in \mathcal{T}(\Sigma,J)$
is written  $[T]$. 
We remark that we have
\begin{align}
\label{equation_kouten_sa}
&[T(d_1)]-[T(d_2)] =(A-\gyaku{A})([T(d_\infty)]-[T(d_0)]),
\end{align}
where 
$d_1$,$d_2$, $d_\infty$ and $d_0$ are differ only 
in an open disk shown in Figure 
\ref{fig_K2}, Figure \ref{fig_K1},
Figure \ref{fig_Kinfi} and Figure \ref{fig_K0}, respectively.

\begin{figure}[htbp]
	\begin{tabular}{rrrr}
	\begin{minipage}{0.25\hsize}
		\centering
		\includegraphics[width=2cm]{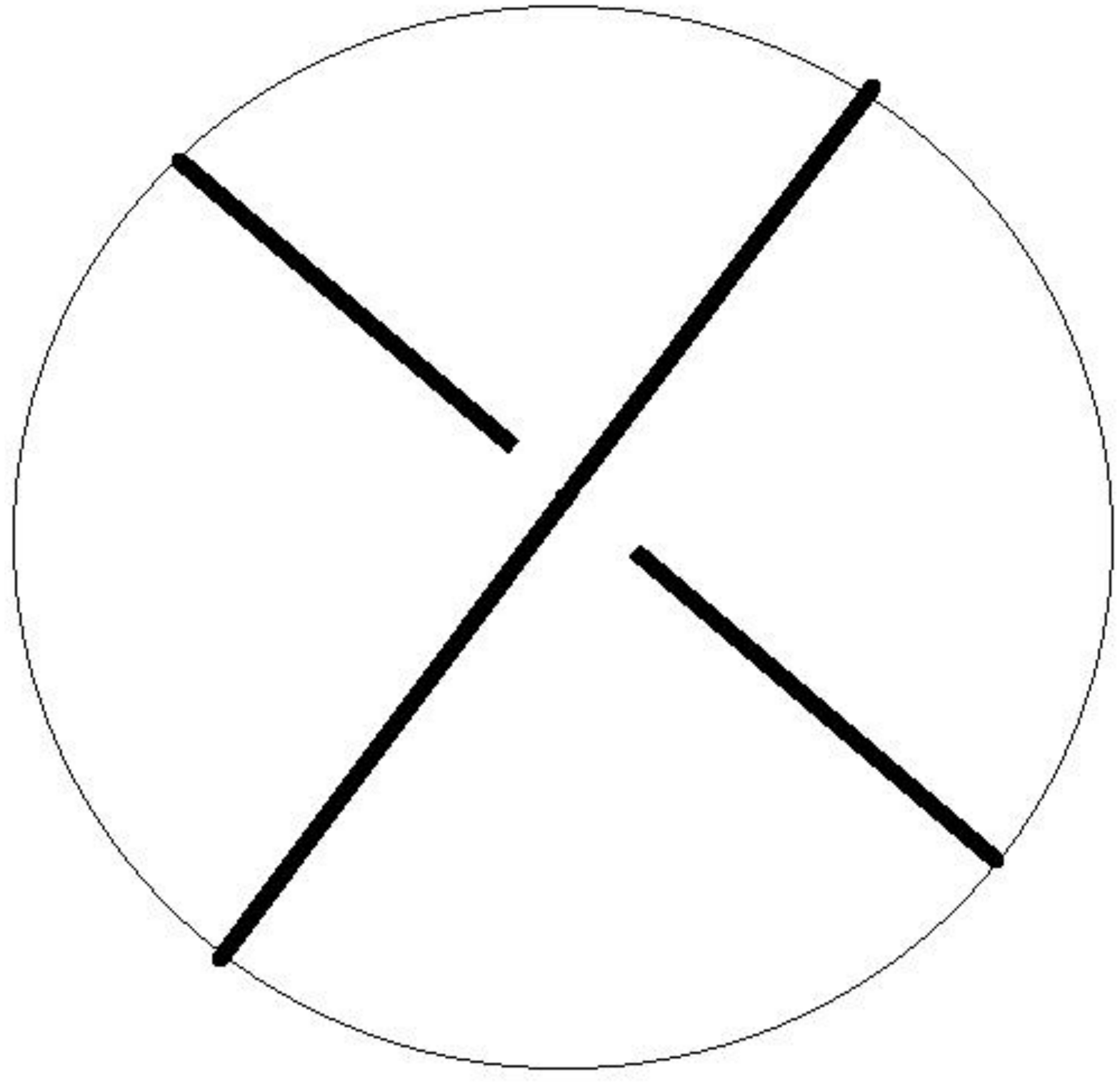}
		\caption{$d_1$}
		\label{fig_K2}
	\end{minipage}
     \begin{minipage}{0.25\hsize}
		\centering
		\includegraphics[width=2cm]{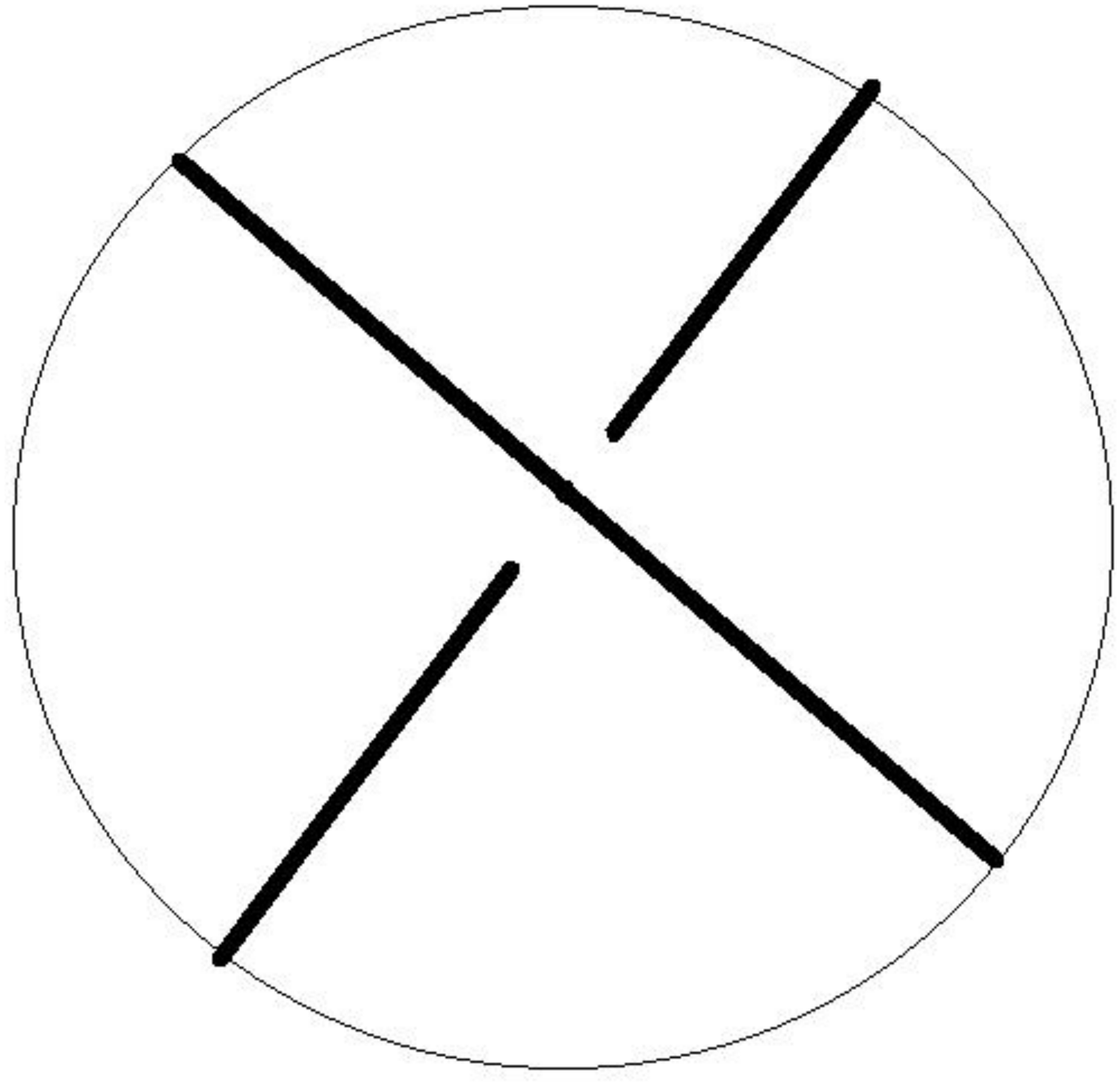}
		\caption{$d_2$}
		\label{fig_K1}
	\end{minipage}
		\begin{minipage}{0.25\hsize}
		\centering
		\includegraphics[width=2cm]{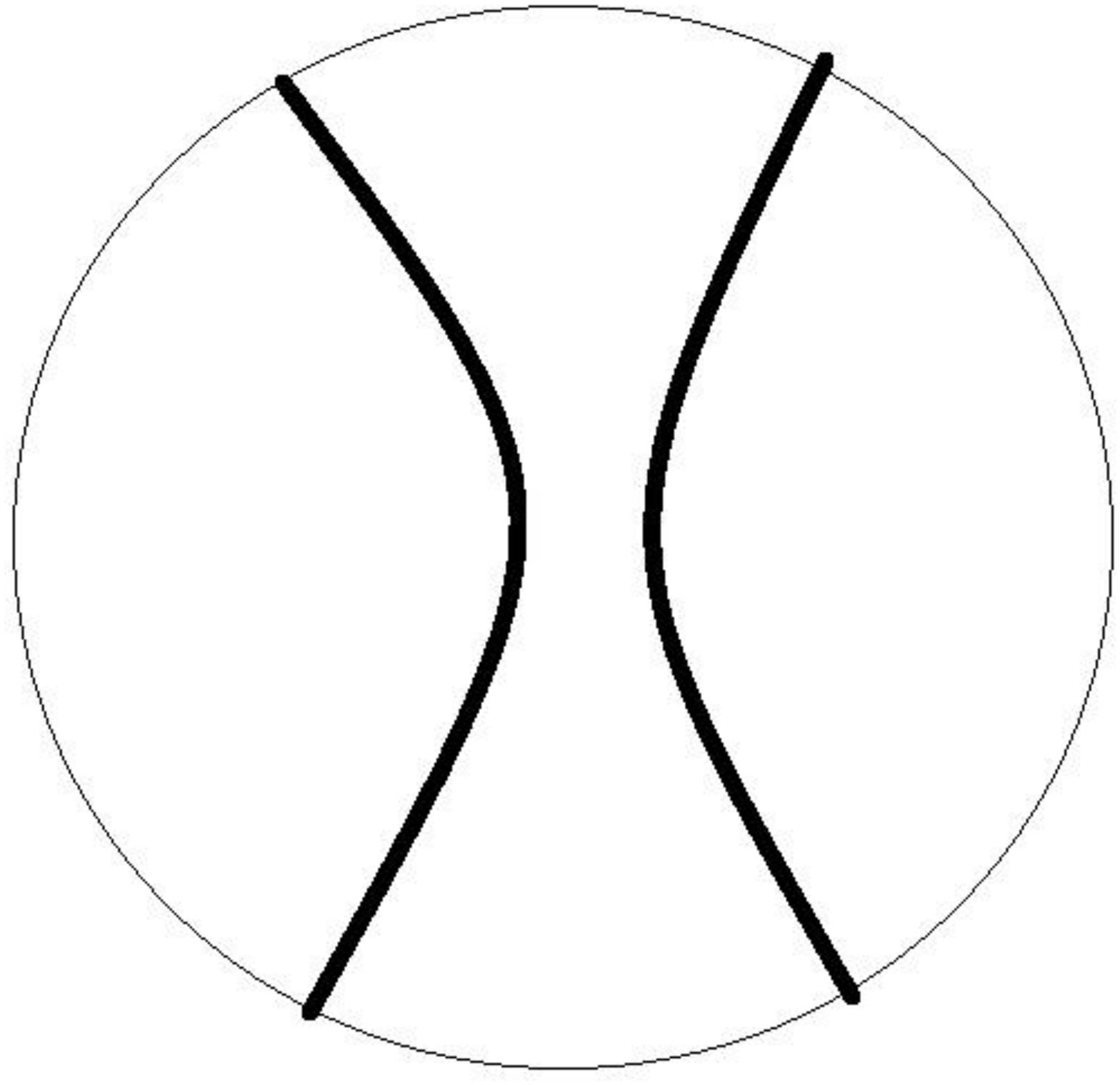}
		\caption{$d_\infty$}
		\label{fig_Kinfi}
	\end{minipage}
		\begin{minipage}{0.25\hsize}
		\centering
		\includegraphics[width=2cm]{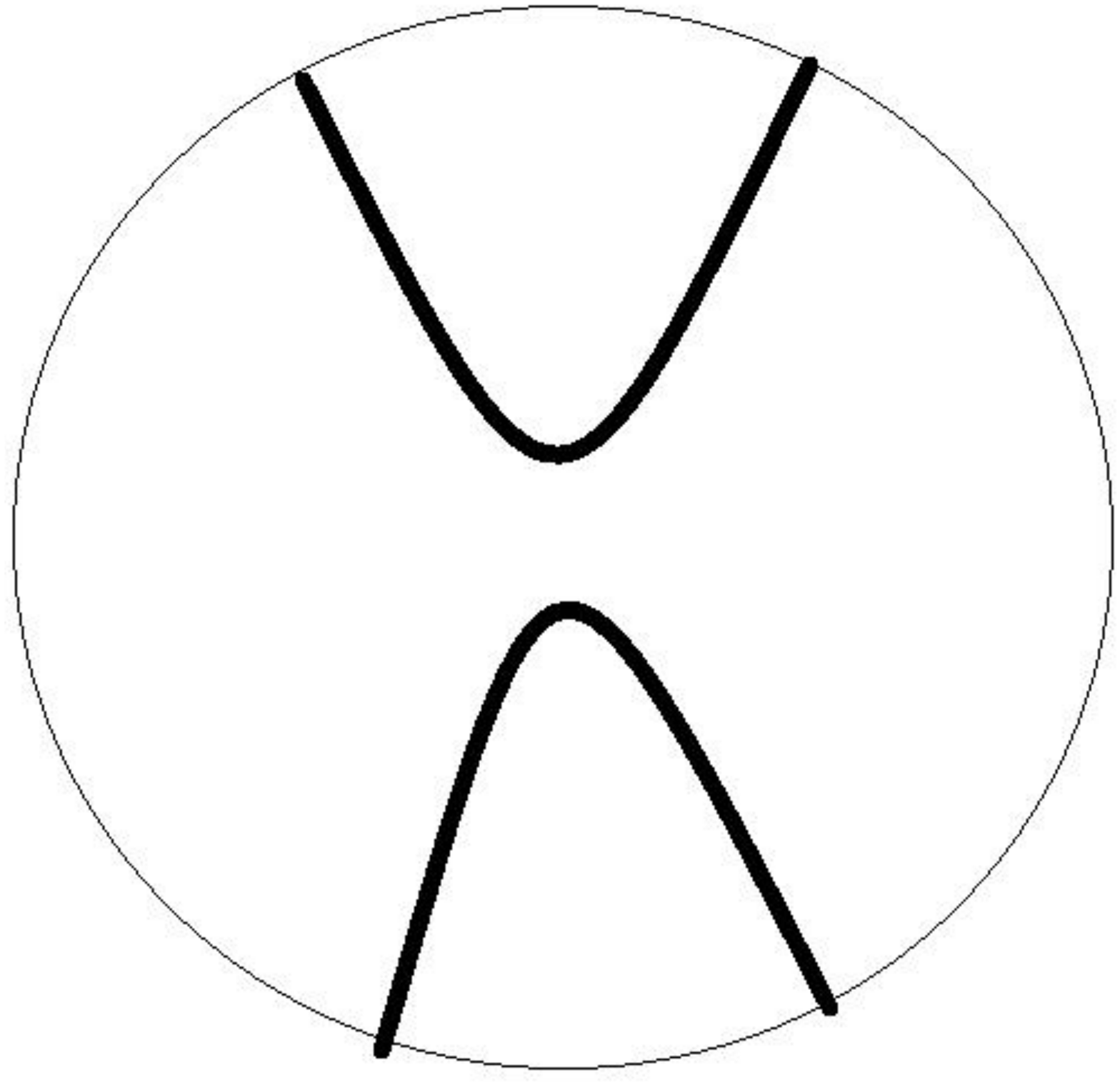}
		\caption{$d_0$}
		\label{fig_K0}
	\end{minipage}
	\end{tabular}
\end{figure}

We denote by $\mathcal{M}(\Sigma)$ the mapping class group of $\Sigma$
fixing the boundary pointwise.
There is a natural action of $\mathcal{M}(\Sigma)$ on $\skein{\Sigma,J}$ \cite{TsujiCSAI} section 2.
The product of $\skein{\Sigma}$ and
the right and left actions of $\skein{\Sigma}$ on $\skein{\Sigma,J}$
are defined by Figure \ref{fig_product_action}, for details, see \cite{TsujiCSAI} 3.1.
The Lie bracket $[,]:\skein{\Sigma} \times \skein{\Sigma} \to \skein{\Sigma}$ is defined by
$[x,y]\defeq \frac{1}{-A+\gyaku{A}} (xy-yx)$.
Furthermore, the action $\sigma()(): \skein{\Sigma} \times \skein{\Sigma,J}
\to \skein{\Sigma,J}$ defined by
$\sigma(x)(z) \defeq \frac{1}{-A+\gyaku{A}}(xz-zx)$
making $\skein{\Sigma,J} $ a $(\skein{\Sigma},[,])$-module with $\sigma$.
For details, see \cite{TsujiCSAI} 3.2.

\begin{figure}
\begin{picture}(300,83)
\put(0,-20){\includegraphics[width=90pt]{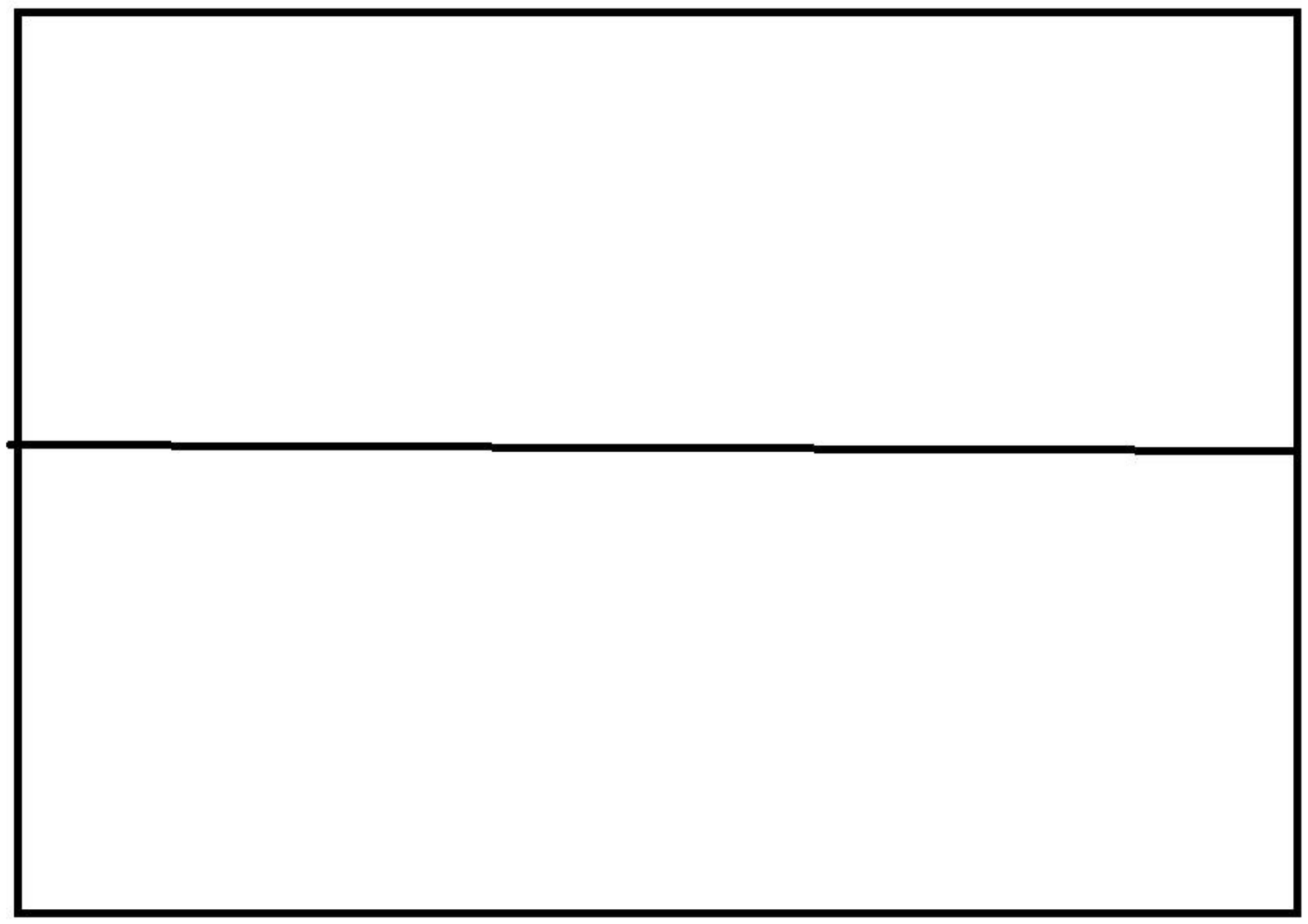}}
\put(100,-20){\includegraphics[width=90pt]{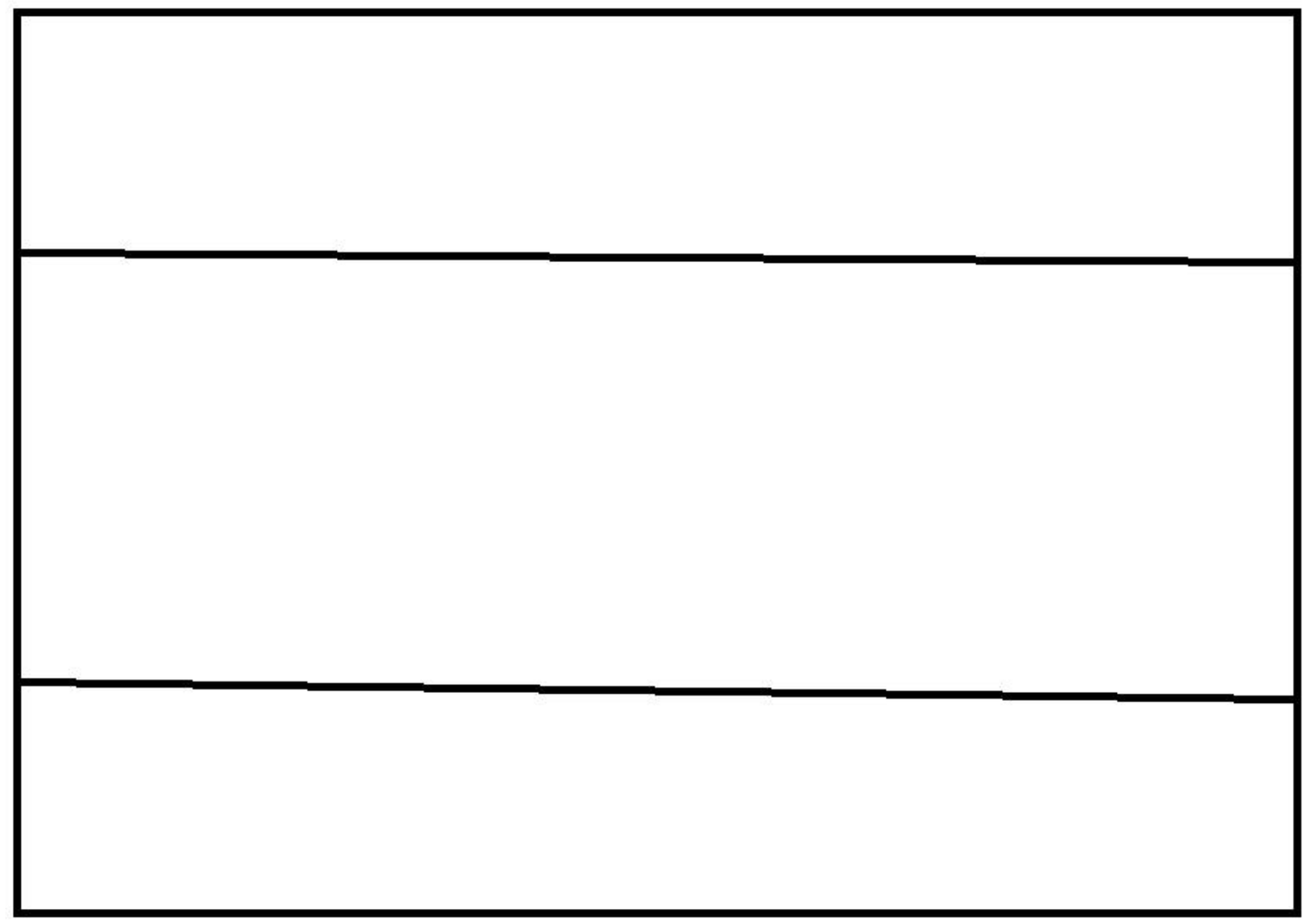}}
\put(200,-20){\includegraphics[width=90pt]{sikaku_3_PNG.pdf}}
\put(10,80){$xy \defeq$}
\put(40,50){$x$}
\put(40,20){$y$}
\put(10,2){$\mathrm{for} \ \  x,y \in \skein{\Sigma}$}
\put(0,10){$0$}
\put(0,55){$1$}
\put(0,32.5){$I$}
\put(40,65){$\Sigma$}

\put(110,80){$xz \defeq$}
\put(140,52){$x$}
\put(140,34){$z$}
\put(110,2){$\mathrm{for} \ \  x\in \skein{\Sigma}$}
\put(110,-10){$\ \ \mathrm{and} \ \ z \in \skein{\Sigma,J}$}
\put(100,10){$0$}
\put(100,55){$1$}
\put(100,32.5){$I$}
\put(140,65){$\Sigma$}

\put(210,80){$zx \defeq$}
\put(240,34){$z$}
\put(240,15){$x$}
\put(210,2){$\mathrm{for} \ \  x \in \skein{\Sigma}$}
\put(210,-10){$\ \ \mathrm{and} \ \ z \in \skein{\Sigma,J}$}
\put(200,10){$0$}
\put(200,55){$1$}
\put(200,32.5){$I$}
\put(240,65){$\Sigma$}

\end{picture}
\caption{The product and the actions}
\label{fig_product_action}
\end{figure}

By \cite{TsujiCSAI} Theorem 3.3, we have the following proposition.

\begin{prop}
\label{prop_map_inj_notcomp}
Let $J$ be a finite subset of $\partial \Sigma$ and $\partial_1, \cdots, \partial_b$
the connected components of $\Sigma$.
If $b \geq 1$ and $\partial_i \cap J \neq \emptyset$
for each $i \in \shuugou{1,2, \cdots,b}$, then
$\mathcal{M}(\Sigma) \to \Aut (\skein{\Sigma,J})$ is injective.

\end{prop}

\subsection{Filtrations and completions of $\skein{\Sigma}$ and $\skein{\Sigma,J}$}
\label{subsection_filtration_1}
The augmentation map $ \epsilon:\skein{\Sigma} \to \Q$ is defined by
$A \mapsto -1$ and $[L] \mapsto (-2)^{\zettaiti{L}}$ for $L \in \mathcal{T}(\Sigma)$
where $\zettaiti{L}$ is the number of components of $L$.
The augmentation map $\epsilon$ is well-defined by \cite{TsujiCSAI} Proposition
3.10.
We consider the topology on $\skein{\Sigma}$ induced by the filtration
$\filtn{(\ker \epsilon)^n}$ and denote its completion
by $\widehat{\skein{\Sigma}} \defeq \comp{i} \skein{\Sigma}/(\ker \epsilon)^i$.
Let $J$ be a finite subset of $\partial \Sigma$.
We also consider the topology on $\skein{\Sigma,J}$ induced  by
the filtration $\filtn{(\ker \epsilon)^n \skein{\Sigma,J}}$ and 
denote its completion by
$\widehat{\skein{\Sigma}} \defeq \comp{i} \skein{\Sigma,J}/
(\ker \epsilon)^i \skein{\Sigma,J}$.
The product of $\skein{\Sigma}$, the right action and the left action
of $\skein{\Sigma}$ on $\skein{\Sigma,J}$, the bracket $[,]$ of 
$\skein{\Sigma}$ and the action $\sigma $ of $\skein{\Sigma}$ on
$\skein{\Sigma,J}$ are continuous. For details, see \cite{TsujiCSAI} Theorem 3.12.

\begin{prop}[\cite{TsujiCSAI} Theorem 5.1]
\label{thm_independent}
Let $\Sigma$ and $\Sigma'$ be two compact connected oriented surfaces,
J and J' finite subsets of $\partial \Sigma$ and $\partial \Sigma'$, respectively. 
We assume there exists a orientation preserving 
diffeomorphism $\mathcal{X}: (\Sigma \times I,
J \times I) \to (\Sigma' \times I, J' \times I)$.
Then we have $\mathcal{X} ((\ker \epsilon)^n  \skein{\Sigma,J}) 
=(\ker \epsilon)^n \skein{\Sigma',J'}$ for each $n$.

\end{prop}

\begin{prop}[\cite{TsujiCSAI} Theorem5.5]
Let $\Sigma$ be a compact connected oriented surface and
$J$ a finite subset of $\partial \Sigma$.
If $\partial \Sigma \neq \emptyset$,
then the natural homomorphism 
$\skein{\Sigma,J} \to \widehat{\skein{\Sigma,J}}$
is injective.
\end{prop}

By this theorem and Proposition \ref{prop_map_inj_notcomp},
we have the following.

\begin{cor}
\label{cor_map_inj}
Let $J$ be a finite subset of $\partial \Sigma$ and $\partial_1, \cdots, \partial_b$
the connected components of $\partial \Sigma$.
If $b \geq 1$ and $\partial_i \cap J \neq \emptyset$
for each $i \in \shuugou{1,2, \cdots,b}$, then
$\mathcal{M}(\Sigma) \to \Aut (\widehat{\skein{\Sigma,J}})$ is injective.
\end{cor}

For a simple closed curve $c$,
we denote 
\begin{equation}
\label{equation_L_c}
L(c) \defeq \frac{-A+\gyaku{A}}{4 \log(-A)} (\arccosh  (-\frac{c}{2}))^2
\end{equation}
where $c$ is also denoted by the element of $\skein{\Sigma}$ represented by 
the knot presented by the simple closed curve $c$.

\begin{thm}[\cite{TsujiCSAI} Theorem 4.1]
\label{thm_Dehn_twist}
Let $J$ be a finite subset of $\partial \Sigma$, $c$  a simple
closed curve and $t_c$ the Dehn twist along $c$.
Then we have
\begin{equation*}
t_c(\cdot) = \exp (\sigma(L(c))) \defeq \sum_{i=0}^\infty \frac{1}{i!} (\sigma(L(c)))^i
\in \Aut (\widehat{\skein{\Sigma,J}}).
\end{equation*}
\end{thm}

\section{The calculations of $\skein{\Sigma} /(\ker \epsilon)^2$}
\label{section_computation}
Let $\Sigma$ be a compact connected oriented surface
as in section \ref{section_definition_review}.

\subsection{the unoriented Goldman Lie algebra}
We recall the Goldman Lie algebra.
We denote by $\hat{\pi} (\Sigma ) =[S^1,\Sigma]$ the homotopy 
set of oriented free loops on $\Sigma$. In other words,
$\hat{\pi} (\Sigma)$ is the set of conjugacy classes of 
$\pi_1 (\Sigma)$.

Let $\Sigma$ be a compact connected oriented surface.
Let $\alpha$ and $\beta$ be oriented immersed loops on $\Sigma$ such that
their intersections consist of transverse double points.
For each $p \in \alpha \cap \beta$, let $\alpha_p \beta_p \in \pi_1 (\Sigma,p)$
be the loop going first along the loop $\alpha$ based at $p$, then going along 
$\beta$ based at $p$.
Also, let $\epsilon (p, \alpha,\beta) \in \shuugou{1,-1}$ be the local intersection
number of $\alpha$ and $\beta$ at $p$. The Goldman bracket of 
$\alpha$ and $\beta$ is defined as 
\begin{equation*}
[\alpha,\beta] \defeq \sum_{p \in \alpha \cap \beta} \epsilon(p,\alpha,\beta)
\zettaiti{\alpha_p \beta_p} \in \Q  \hat{\pi}(\Sigma).
\end{equation*}
Here we denote by $\zettaiti{\cdot}: \pi_1 (\Sigma) \to \hat{\pi}(\Sigma)$
the natural projection, and we also denote by 
$\zettaiti{\cdot}: \Q \pi_1 (\Sigma) \to \Q \hat{\pi}(\Sigma)$
its $\Q$-linear extension. The free $\Q$-vector space 
$\Q \hat{\pi} (\Sigma)$ spanned by the set $\hat{\pi}(\Sigma)$ equipped with this bracket
is a Lie algebra. See \cite{Goldman}.

According to \cite{Goldman} there is a Lie algebra similar to
$\Q \hat{\pi}(\Sigma)$  but based on homotopy classes of unoriented
loops. 
The map $\hat{\pi}(\Sigma) \to
\hat{\pi}(\Sigma),\zettaiti{a} \mapsto \zettaiti{\gyaku{a}}$
which reverses the orientation of oriented loops extends to a Lie algebra 
automorphism $\Q \hat{\pi}(\Sigma) \to
\Q \hat{\pi}(\Sigma)$ denoted by $\upsilon$.
Clearly, $(\upsilon+\id)(\Q \hat{\pi}(\Sigma))$ is a Lie subalgebra of
$\Q \hat{\pi}(\Sigma)$ which is a free module over the set of
all $\zettaiti{a}+\upsilon(\zettaiti{a}) \defeq (a)_\square$, $a \in \pi_1(\Sigma)$.
The following formula given in \cite{Goldman}
computes the bracket of generators $(a)_\square$ and
$(b)_\square$
\begin{equation*}
[(a)_\square,(b)_\square] = \sum_{p \in \zettaiti{a} \cap \zettaiti{b}}
\epsilon(p,\zettaiti{a},\zettaiti{b}) ((a_pb_p)_\square-(a_p\gyaku{b_p})_\square).
\end{equation*}
We denote simply $\shuugou{(r)_\square | r \in \pi_1 (\Sigma)}$ and
$(\upsilon+\id)(\Q \hat{\pi}(\Sigma))$ by $\pi_\square (\Sigma)$ and 
$\Q \pi_\square (\Sigma)$.

\subsection{The homomorphism $\kappa: \Q \pi_\square (\Sigma)
\to (\ker \epsilon)/(\ker \epsilon)^2 $}
\label{subsection_kappa}
In this subsection, we construct a Lie algebra homomorphism 
$\kappa: \Q \pi_\square (\Sigma) \to (\ker \epsilon)/(\ker \epsilon)^2 $.
This is an analogy of \cite{Turaev} Theorem 3.3.

For $x \in \pi_1(\Sigma)$,
we define $\kukakko{x} \in (\ker \epsilon)/(\ker \epsilon)^2 $ 
by $[L_x]+2-3 w(L_x) (A-\gyaku{A})$ using $L_x \in \mathcal{T}(\Sigma)$
with $p_1 (L_x) =\zettaiti{x}$ where
the writhe $w(L_x)$ is the sum of the signs of the crossing of a diagram 
presenting $L_x$.

\begin{lemm}
The map $\kukakko{\cdot}:\pi_1 (\Sigma) \to (\ker \epsilon)/(\ker \epsilon)^2$
is well-defined.
\end{lemm}

\begin{proof}
Let $d_1$, $d_2$, $d_3$ and $d_4$ be four link diagrams in $\Sigma$
which are differ only in an open disk
shown in Figure 
\ref{fig_K2}, Figure \ref{fig_K1},
Figure \ref{fig_Kinfi} and Figure \ref{fig_K0}, respectively
such that $\zettaiti{T(d_1)} =\zettaiti{T(d_2)} = \zettaiti{T(d_\infty)} =1$
and $\zettaiti{T(d_0)} =2$. We denote $L_i \defeq T(d_i)$ for $i \in \shuugou{
1,2,0, \infty}$.
We have $w(L_1)-w(L_2) =2$.
By the equation (\ref{equation_kouten_sa}), we have 
\begin{align*}
&([L_1]+2-3w(L_1)(A-\gyaku{A}))-([L_2]+2-3w(L_2)(A-\gyaku{A})) \\
&=(A-\gyaku{A})([L_\infty]-[L_0])-3(w(L_1)-w(L_2))(A-\gyaku{A}) \\
&=(A-\gyaku{A})([L_\infty]-[L_0]-6).
\end{align*}
Since $\epsilon(A-\gyaku{A})=\epsilon([L_\infty]-[L_0]-6)=0$, we have
$([L_1]+2-3w(L_1)(A-\gyaku{A}))-([L_2]+2-3w(L_2)(A-\gyaku{A})) \in (\ker \epsilon)^2$.

Let $d$ and $d'$ be two knot diagrams in $\Sigma$ 
 which are differ
only in an open disk shown in the figure.
We denote $L \defeq  T(d)$ and $T' \defeq T(d')$.

\begin{picture}(300,80)
\put(0,0){\includegraphics[width=60pt]{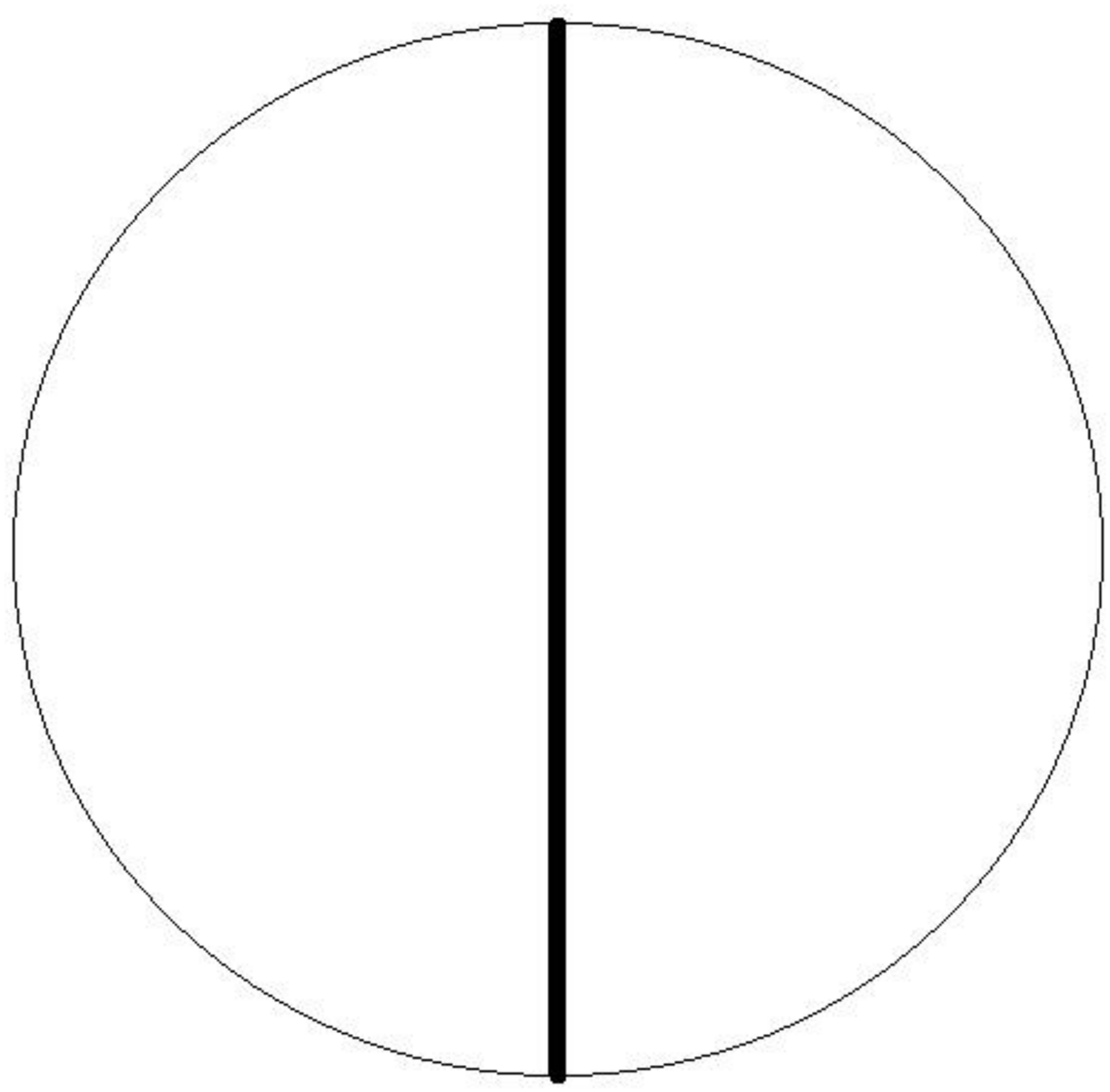}}
\put(60,0){\includegraphics[width=60pt]{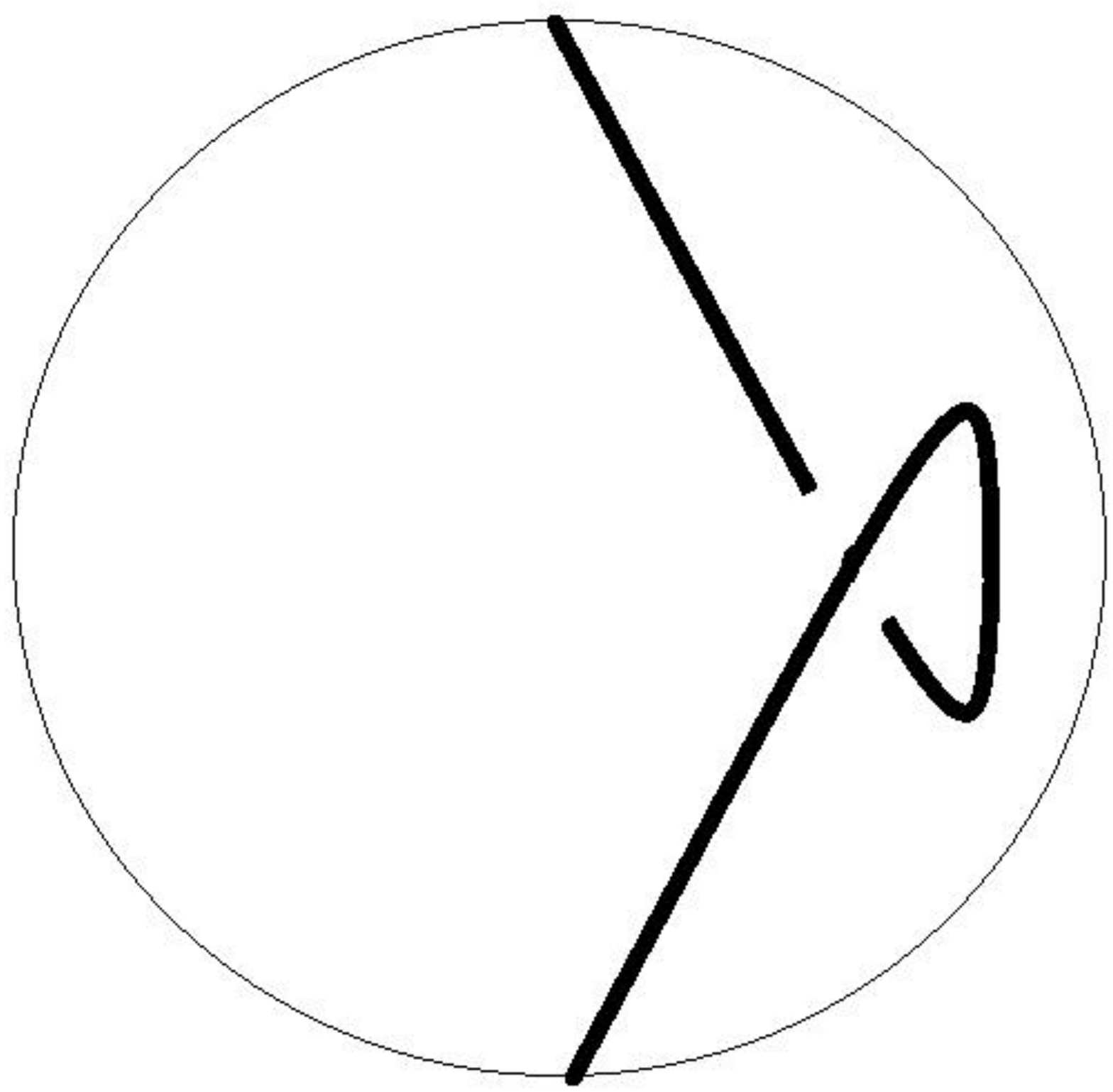}}
\put(20,70){$d$}
\put(80,70){$d'$}
\end{picture}

Then $w(L)-w(L')=-1$, and so
\begin{align*}
&([L]+2-3w(L)(A-\gyaku{A}))-([L']+2-3w(L')(A-\gyaku{A})) \\
&=(A^3+1)[L']-3(w(L_1)-w(L_2))(A-\gyaku{A}) \\
&=(A+1)((A^2-A+1)[L']+3(1-\gyaku{A})).
\end{align*}
Since $\epsilon(A+1) =\epsilon((A^2-A+1)[L']+3(1-\gyaku{A}))=0$,
we have $([L]+2-3w(L)(A-\gyaku{A}))-([L']+2-3w(L')(A-\gyaku{A})) \in 
(\ker \epsilon)^2$.
 
This finishes the proof.
\end{proof}

Here we remark that $\kukakko{1} =0$ 
for the identity $1 \in \pi_1 (\Sigma)$
and $\kukakko{x} =\kukakko{yx\gyaku{y}}$ for $x,y \in \pi_1(\Sigma)$.
We also denote the $\Q$-linear extension of the map $\kukakko{\cdot}$
by $\kukakko{\cdot} :\Q \pi_1 (\Sigma) \to \ker{\epsilon}/(\ker \epsilon)^2$.

Since $[ \ker \epsilon, (\ker \epsilon)^2] \subset (\ker \epsilon)^2$
and $[ (\ker \epsilon)^2, \ker \epsilon] \subset (\ker \epsilon)^2$
\cite{TsujiCSAI} Lemma 3.11,
$[,]:\skein{\Sigma} \times \skein{\Sigma} \to \skein{\Sigma}$ induces 
$[,]:(\ker \epsilon)/(\ker \epsilon)^2 \times 
(\ker \epsilon)/(\ker \epsilon)^2\to (\ker \epsilon)/(\ker \epsilon)^2$.

\begin{thm}
\label{thm_kappa}
This $\Q$-linear map $\kappa: \Q \pi_\square (\Sigma)
\to (\ker \epsilon)/(\ker \epsilon)^2, (x)_\square \mapsto -\kukakko{x}$ is 
a Lie algebra homomorphism.
\end{thm} 

\begin{proof}
Let $d_a$ and $d_b$  be two knot diagrams in $\Sigma$
whose intersections  consist of transverse double points 
$P_1,P_2, \cdots, P_m$.
We fix orientations of $d_a$ and $d_b$.
Let $\alpha$ and $\beta$ be two elements of $[S^1,\Sigma]$ such that
$\alpha = p_1(d_a)$ and $\beta = p_1(d_b)$.
We write simply $\epsilon_i \defeq \epsilon(P_i,d_a,d_b) $ for $i=1,2, \dots,m$.
Let $\alpha_\star$ and $\beta_\star$ be two elements of $\pi_1(\Sigma)$ 
such that $\zettaiti{\alpha_\star} =\alpha$ and $\zettaiti{\beta_\star}= \beta$.

For $i=1,2, \cdots, m$, let $d(1,i)$ and $d(-1,i)$ be two tangle diagrams
satisfying the following contitions.
\begin{itemize}
\item The two tangle diagrams $d(1,i)$ and $d(-1,i)$ equal
$d_a \cup d_b$ with the same height-information as
$d_a$ and $d_b$ except for the neighborhoods of the intersections
of $d_a$ and $d_b$.
\item The branches of $d(1,i)$ and $d(-1,i)$ in the neighborhood of $P_j$
belonging to $d_a$ are over crossings for
$j=1, \cdots, i-1$.
\item The branches of $d(1,i)$ and $d(-1,i)$ in the neighborhood of $P_j$
belonging to $d_a$ are over crossings for 
$j=i+1, \cdots,m$.
\item The two tangle diagrams $d(1,i)$ and $d(-1,i)$  are as shown in Figure
\ref{fig_d_and_d_1}  and Figure \ref{fig_d_and_d_2}, respectively, 
in the neighborhood of $P_i$.
\end{itemize}

\begin{figure}[htbp]
	\begin{tabular}{cc}
	\begin{minipage}{0.33\hsize}
		\centering
\includegraphics[width=3cm]{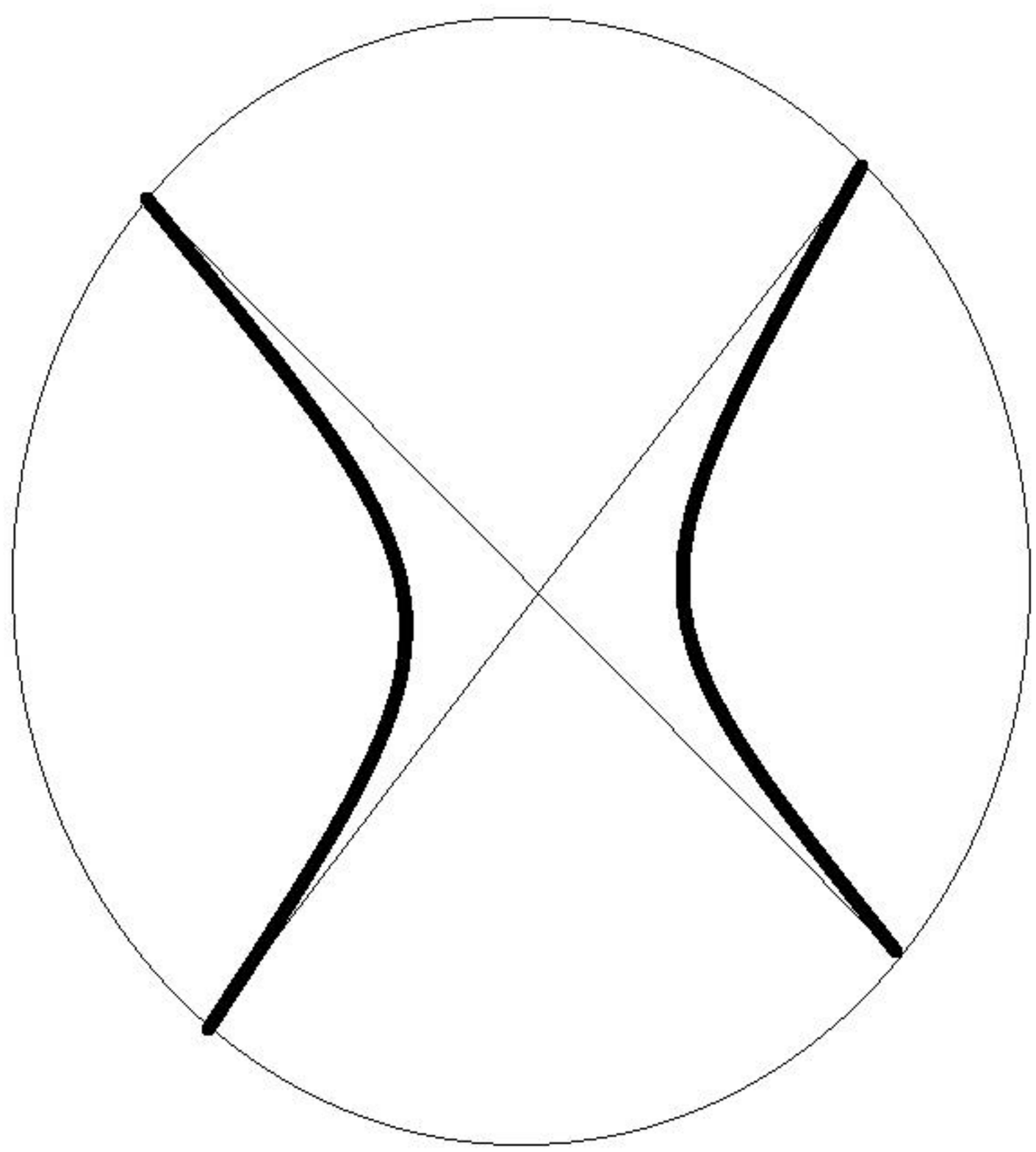}
\put(-30,30){$d_b$}
\put(-30,83){$d_a$}
		\caption{}
		\label{fig_d_and_d_1}
	\end{minipage}
     \begin{minipage}{0.33\hsize}
		\centering
		\includegraphics[width=3cm]{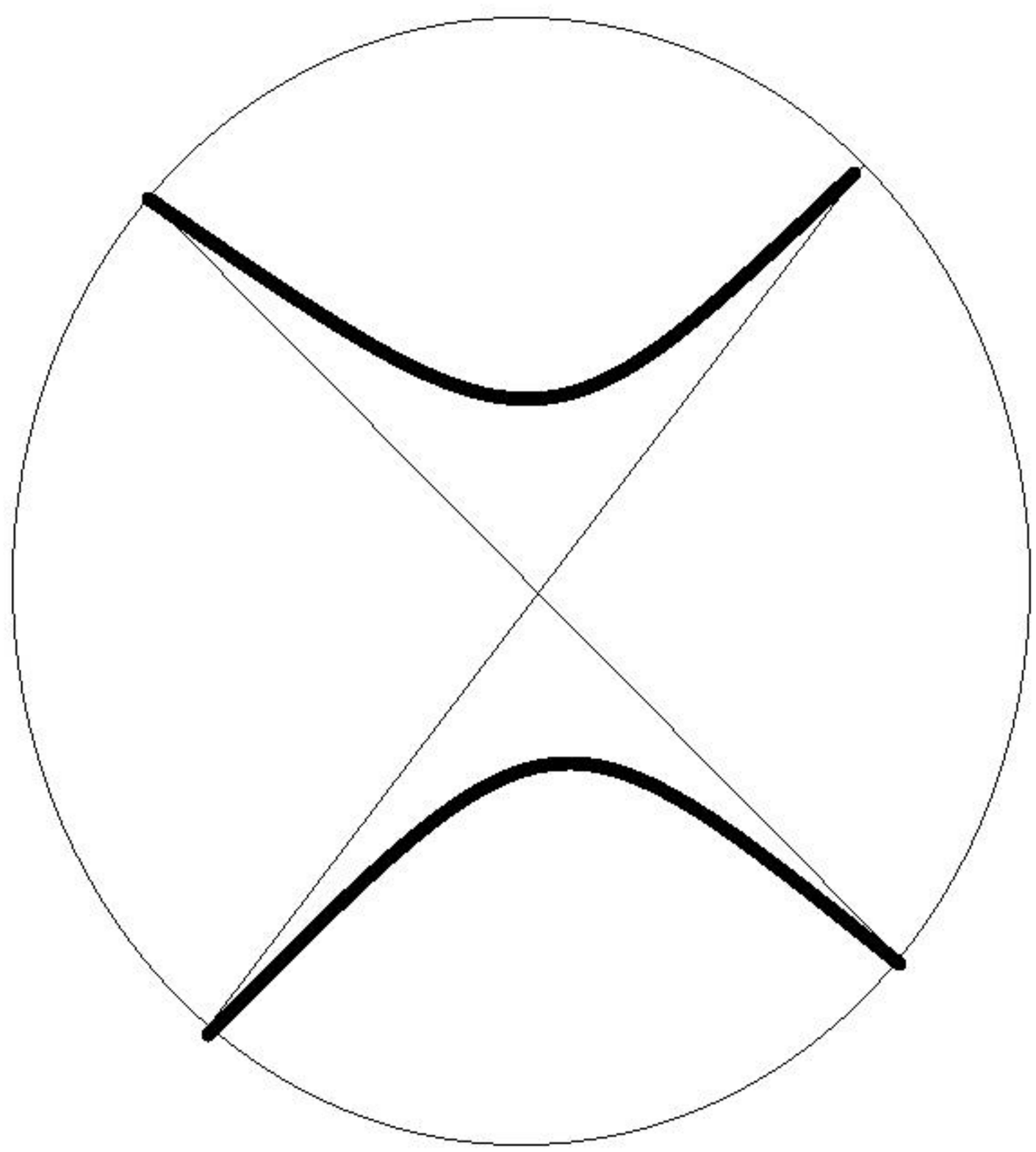}
\put(-30,30){$d_b$}
\put(-30,85){$d_a$}
		\caption{}
		\label{fig_d_and_d_2}
	\end{minipage}
	\end{tabular}
\end{figure}

We denote $T(1,i) \defeq T(d(1,i)) $ and $T(-1,i) \defeq T(d(-1,i))$
for $ i=1, \cdots, m$. 

Using the equation (\ref{equation_kouten_sa}), we have 
\begin{equation*}
[[T_a],[T_b]] =-\sum^m_{i=1}([T(1,i)]-[T(-1,i)]).
\end{equation*}

If $\epsilon_i =1$, we have 
\begin{align*}
\kukakko{\alpha_{P_i}\beta_{P_i}} &= [T(1,i)]+2-3(w(T)+w(T')-\epsilon_1-
\dots-\epsilon_{i-1}+\epsilon_{i+1}+ \cdots +\epsilon_{m})(A-\gyaku{A}), \\
\kukakko{ \alpha_{P_i}\gyaku{\beta_{P_i}}} &= [T(-1,i)]+2-3(w(T)+w(T')+\epsilon_1+
\dots+\epsilon_{i-1}-\epsilon_{i+1}- \cdots -\epsilon_{m})(A-\gyaku{A}). \\
\end{align*}

If $\epsilon_i =-1$, we have 
\begin{align*}
\kukakko{ \alpha_{P_i}\gyaku{\beta_{P_i}}} &= [T(1,i)]+2-3(w(T)+w(T')+\epsilon_1+
\dots+\epsilon_{i-1}-\epsilon_{i+1}- \cdots -\epsilon_{m}), \\
\kukakko{\alpha_{P_i}\beta_{P_i}} &= [T(-1,i)]+2-3(w(T)+w(T')-\epsilon_1-
\dots-\epsilon_{i-1}+\epsilon_{i+1}+ \cdots +\epsilon_{m}). \\
\end{align*}

Hence, we have
\begin{align*}
&[-\kukakko{\alpha_\star},-\kukakko{\beta_\star}] \\
&=[T_a+2-3w(T_a)(A-\gyaku{A}),T_b+2-3w(T_b)(A-\gyaku{A})] \\
&=[T_a,T_b] \\
&=-\sum^m_{i=1}([T(1,i)]-[T(-1,i)]) \\
&=-\sum^m_{i=1}\epsilon_i(\kukakko{\alpha_{P_i}\beta_{P_i}}-\kukakko{
 \alpha_{P_i}\gyaku{\beta_{P_i}}} +6(-\epsilon_1-
\dots-\epsilon_{i-1}+\epsilon_{i+1}+ \cdots +\epsilon_{m})(A-\gyaku{A})) \\
&=-\sum^m_{i=1}\epsilon_i(\kukakko{\alpha_{P_i}\beta_{P_i}}-\kukakko{
 \alpha_{P_i}\gyaku{\beta_{P_i}}}) +6(A-\gyaku{A})\sum_{i<j}(\epsilon_i \epsilon_j-
\epsilon_j \epsilon_i) \\
&=-\sum^m_{i=1}\epsilon_i(\kukakko{\alpha_{P_i}\beta_{P_i}}-\kukakko{
 \alpha_{P_i}\gyaku{\beta_{P_i}}}).
\end{align*}
This finishes the proof.
\end{proof}

The following proposition plays a key role in  our calculations
of $\skein{\Sigma}/(\ker \epsilon)^2$.

\begin{prop}
For $x$ and $y \in \pi_1 (\Sigma)$, we have $2\kukakko{x} +2\kukakko{y} =
\kukakko{xy} +\kukakko{x\gyaku{y}}$.
\end{prop}

\begin{proof}
We fix the base point $*$ of $\pi_1 (\Sigma)$ in $\Sigma \backslash
\partial \Sigma$.
Let $d_a$ and $d_b$ be knot diagrams in $\Sigma$
whose interesections consist of transverse double points and
which transverse at $*$ as shown in Figure \ref{fig_dxy_1}.

\begin{figure}[htbp]
	\begin{tabular}{cc}
	\begin{minipage}{0.33\hsize}
		\centering
\includegraphics[width=3cm]{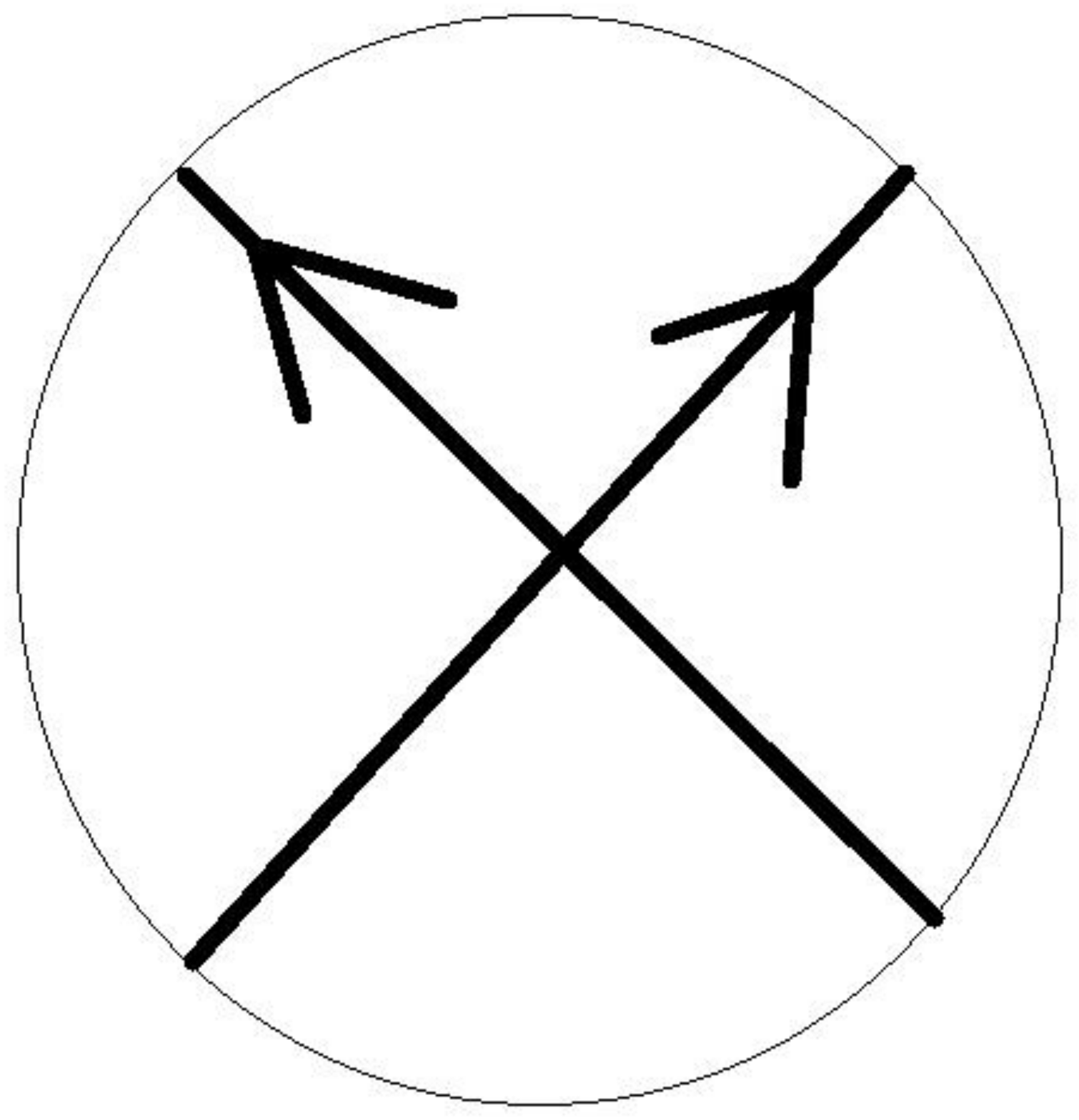}
\put(-30,25){$d_b$}
\put(-33,77){$d_a$}
\put(-39,50){$*$}
		\caption{}
		\label{fig_dxy_1}
	\end{minipage}
\begin{minipage}{0.33\hsize}
		\centering
\includegraphics[width=3cm]{smooth_1_PNG.pdf}
\put(-30,30){$d_b$}
\put(-30,85){$d_a$}
\put(-43,50){$*$}
		\caption{}
		\label{fig_dxy_2}
	\end{minipage}
     \begin{minipage}{0.33\hsize}
		\centering
		\includegraphics[width=3cm]{smooth_2_PNG.pdf}
\put(-30,30){$d_b$}
\put(-30,85){$d_a$}
\put(-43,50){$*$}
		\caption{}
		\label{fig_dxy_3}
	\end{minipage}
	\end{tabular}
\end{figure}

Let $d_{ab}$ and $d_{a\gyaku{b}}$ be two knot diagrams satisfying 
the following conditions.

\begin{itemize}
\item The two knot diagrams $d_{ab}$ and $d_{a\gyaku{b}}$
equal 
$d_a \cup d_b$ with the same height-information as $d_a$ and $d_b$
except for the neighborhood of the intersections of $d_a$ and $d_b$.
\item The branches of $d_{ab} $ and $d_{a \gyaku{b}}$ in the neighborhood
of any point of $d_a  \cap d_b \backslash \shuugou{*}$ belonging to
$d_a$ are over crossings.
\item The two tangle diagrams $d_{ab}$ and $d_{a\gyaku{b}}$
are as shown in Figure \ref{fig_dxy_2} and Figure \ref{fig_dxy_3},
respectively.
\end{itemize}

We denote $T_i \defeq T(d_i)$ for $i \in \shuugou{a,b,ab, a\gyaku{b}}$.
We remark that $w(T_{ab}) +w(T_{a\gyaku{b}}) =
2(w(T_a) +w(T_b))$.

Hence we have

\begin{align*}
&([T_a]+2)([T_b]+2) \\
&=[T_a][T_b] +2[T_a]+2[T_b]+4 \\
&=A[T_{ab}] +\gyaku{A}[T_{a\gyaku{b}}]+2[T_a]+2[T_b]+4 \\
&=(A+1)([T_{ab}]+\gyaku{A}[T_{a\gyaku{b}}])
-[T_{ab}]-[T_{a\gyaku{b}}]+2[T_a]+2[T_b]+4 \\
&=(A+1)([T_{ab}]+\gyaku{A}[T_{a\gyaku{b}}]) \\
&-([T_{ab}]+2-3(A-\gyaku{A})w(T_{ab}))
-([T_{a\gyaku{b}}]+2-3(A-\gyaku{A})w(T_{a\gyaku{b}})) \\
&+2([T_{a}]+2-3(A-\gyaku{A})w(T_{a}))
+2([T_{b}]+2-3(A-\gyaku{A})w(T_{b}))
\end{align*}

This proves the proposition.

\end{proof}

\subsection{The calculation lemma}

In this subsection, we check some equations.

\begin{lemm}
\label{lemm_calc_1}
For $a$, $b$, $c$ and $d \in \pi_1(\Sigma)$, we have
\begin{align}
&\kukakko{(a-1)(b-1)(c-1)} = \kukakko{-(b-1)(a-1)(c-1)}, \label{equation_calc_1}\\
&\kukakko{(a-1)(b-1)(c-1)(d-1)} =0, \label{equation_calc_2}\\
&\kukakko{(a-1)(b-1)(b-1)} = 0, \label{equation_calc_3}\\
&\kukakko{[a,b]c-c} = 2\kukakko{(a-1)(b-1)(c-1)}, \label{equation_calc_4}\\
&\kukakko{(a-1)(a-1)}=2\kukakko{a}. \label{equation_calc_5}
\end{align}
\end{lemm}

\begin{proof}
We have

\begin{align*}
&\kukakko{abc}  \\
&=\kukakko{-\gyaku{b}\gyaku{a}c+2ab+2c}  \\
&=\kukakko{\gyaku{b}ac-2\gyaku{b}c-2a+2ab+2c}  \\
&=\kukakko{-bac+2ac+2b+2bc-4b-4c-2a+2ab+2c} \\
&=\kukakko{-bac+2ab+2bc+2ca-2a-2b-2c}.
\end{align*}

Hence we have 
$\kukakko{(a-1)(b-1)(c-1)} = \kukakko{-(b-1)(a-1)(c-1)}$.
This proves the equation (\ref{equation_calc_1}). By this equation, we have

\begin{align*}
&\kukakko{(a-1)(b-1)(c-1)(d-1)} \\
&=-\kukakko{(a-1)(d-1)(b-1)(c-1)}\\
&=\kukakko{(a-1)(b-1)(d-1)(c-1)}\\
&=-\kukakko{(a-1)(b-1)(c-1)(d-1)}.
\end{align*}

This proves the equation (\ref{equation_calc_2}).
By equation (\ref{equation_calc_1}). we have
\begin{align*}
&\kukakko{(a-1)(b-1)(b-1)} =-\kukakko{(a-1)(b-1)(b-1)}.
\end{align*}
This proves the equation (\ref{equation_calc_3}).
By the equation (\ref{equation_calc_1}), the equation (\ref{equation_calc_2})
and the equation (\ref{equation_calc_3}), we have

\begin{align*}
&\kukakko{[a,b]c-c} \\
&=\kukakko{(ab-ba)(\gyaku{a}\gyaku{b}c-1)+ab-ba} \\
&=\kukakko{(ab-ba)(\gyaku{a}\gyaku{b}c-1)} \\
&=\kukakko{((a-1)(b-1)-(b-1)(a-1))(\gyaku{a}\gyaku{b}c-1)} \\
&=\kukakko{2(a-1)(b-1)(\gyaku{a}\gyaku{b}c-1)} \\
&=-2\kukakko{(a-1)(b-1)(a-1)}-2\kukakko{(a-1)(b-1)(b-1)}
+2 \kukakko{(a-1)(b-1)(c-1)} \\
&=2\kukakko{(a-1)(b-1)(c-1)}\\.
\end{align*}
This proves the equation (\ref{equation_calc_4}).
We have
\begin{align*}
&\kukakko{a^2-2a+1}
=\kukakko{-1+2a+2a-2a+1}=2\kukakko{a}.
\end{align*}
This proves the equation (\ref{equation_calc_5}).
This finishes the proof.
\end{proof}

We denote $H_1 \defeq H_1(\Sigma,\Q)=\Q \otimes \pi_1(\Sigma)/[\pi_1(\Sigma),\pi_1(\Sigma)]$.
\begin{df}
A $\Q$-linear map $\lambda:H_1 \wedge H_1 \wedge H_1
\to \ker \epsilon/(\ker \epsilon)^2$
is defined by $[a] \wedge [b] \wedge [c]  \defeq \kukakko{(a-1)(b-1)(c-1)}$
for $a$, $b$ and $c \in \pi_1(\Sigma)$
where we denote the third exterior of $H_1$
by $H_1 \wedge H_1 \wedge H_1$.
\end{df}

By the equation (\ref{equation_calc_1}) and the equation (\ref{equation_calc_2}), we have $\lambda$ is 
well-defined.

Since the $\Q$ vector space $\skein{\Sigma}/(\ker \epsilon)^2$ is generated by
$\shuugou{1, A+1} \cup \kappa(\Q \pi_\square (\Sigma))$
as a $\Q$ vector space, 
we have the following 
by teh equation (\ref{equation_calc_1}), the equation (\ref{equation_calc_2}), the
equation (\ref{equation_calc_4})
and the equation (\ref{equation_calc_5}).

\begin{lemm}
\label{lemm_generator}
We assume $\pi_1(\Sigma)$ is generated by $\shuugou{x_1,x_2,\cdots,x_M}$.
Then $\skein{\Sigma}/(\ker \epsilon)^2$ is generated by
\begin{align*}
\shuugou{1} \cup \shuugou{A+1} \cup \shuugou{\kukakko{x_i,x_j}|i \leq j} \cup
\shuugou{\kukakko{x_i,x_j,x_k}|i < j < k}
\end{align*}
as a $\Q$ vector space.
Here we denote $\kukakko{x_i,x_j} \defeq \kukakko{(x_i-1)(x_j-1)}$
and $\kukakko{x_i,x_j,x_k} \defeq \kukakko{(x_i-1)(x_j-1)(x_k-1)}$.

\end{lemm}

\section{A basis of $\skein{\Sigma}/(\ker \epsilon)^2$}
\label{section_basis}

In this section, we will prove the following.

\begin{thm}
\label{thm_basis}
Let $\Sigma$ be a connected compact oriented surface with non-empty boundary.
We assume $\pi_1(\Sigma)$ is freely generated by $\shuugou{x_1,x_2,\cdots,x_M}$.
Then 
\begin{align*}
\shuugou{1} \cup \shuugou{A+1} \cup \shuugou{\kukakko{x_i,x_j}|i \leq j} \cup
\shuugou{\kukakko{x_i,x_j,x_k}|i < j < k}
\end{align*}
is a basis of $\skein{\Sigma}/(\ker \epsilon)^2$
as a $\Q$ vector space.
\end{thm}

To prove this theorem, we introduce a bilinear form of 
the Kauffman bracket skein algebra.
Let $\xi$ be an element of the mapping class group 
$\mathcal{M} (\Sigma)$ represented by a diffeomorphism $\chi$ such that
$S^3 \simeq \Sigma \times I/ \sim_\xi$,
where the equivalence relation $\sim_\xi$ is
generated by $(x,0)  \sim_\xi (\chi(x),1) $ for $x \in \Sigma$ and
$(x,t) \sim_\xi (x,1-t)$ for $x  \in \partial \Sigma$ and $t \in I$.
We remark that for any compact connected oriented surface $\Sigma$
with non-empty boundary
there exists an element of $\mathcal{M} (\Sigma)$ satisfying the above
condition.
The embedding $e_\xi :\Sigma \times I \to \Sigma \times I/\sim_\xi \simeq S^3:
(x,t) \mapsto (x, \frac{t}{2})$ induces a $\Q[A^{\pm1}]$-module homomorphism
$\vartheta_\xi : \skein{\Sigma} \to \Q [A^{\pm1}]:
[L] \mapsto \mathcal{K} (e_\xi(L))$
 where $\mathcal{K}$ is Kauffman bracket.
In this paper, the Kauffman bracket is defined by $[L] =
\mathcal{K}(L) [\emptyset]$ for $L
\in$ (the set of unoriented framed links in $S^3$) $=
\mathcal{T} (I \times I)$. We remark that
$\vartheta_\xi (xy) =\vartheta_\xi (\xi(y)x)$.
By \cite{TsujiCSAI} Lemma 5.7 and 
Lemma 5.8, we have the followings.

\begin{prop}
For any $x \in \skein{\Sigma}\backslash \shuugou{0}$, there exists $y$ such that 
$\vartheta_\xi (xy) \neq 0$.

\end{prop}

\begin{lemm}
\label{lemm_e_filt}
We have $\vartheta_\xi
 ((\ker \epsilon)^N) \subset ((A+1)^N)\Q [A,\gyaku{A}]$ for any $N$.
\end{lemm}

Let $\Sigma$ be a connected compact oriented surface of genus $0$ with
$b+1$ boundary components.
Let $r_i$ be an element of $\pi_1 (\Sigma)$ as in Figure \ref{figure_r_i} 
for $1 \leq i \leq b$
and $c_{i_1 i_2 \cdots i_j}$ a simple closed curve presenting by
$\zettaiti{r_{i_1} \cdots r_{i_j}}$ for $1 \leq i_1<i_2< \cdots<i_j \leq b$.
We remark that $\Sigma \times I/ \sim_{t_{c_1}t_{c_2} \cdots t_{c_b}} \simeq
S^3$ and that  the embedding $e_{t_{c_1}t_{c_2} \cdots t_{c_b}}$
is as in Figure \ref{figure_e}.
We simply denote $\vartheta \defeq \vartheta_{t_{c_1}t_{c_2} \cdots t_{c_b}}$.
We remark that $\vartheta (xy) =\vartheta(yx)$.

\begin{figure}
\begin{flushleft}
\begin{picture}(200,150)
\put(0,-160){\includegraphics[width=360pt]{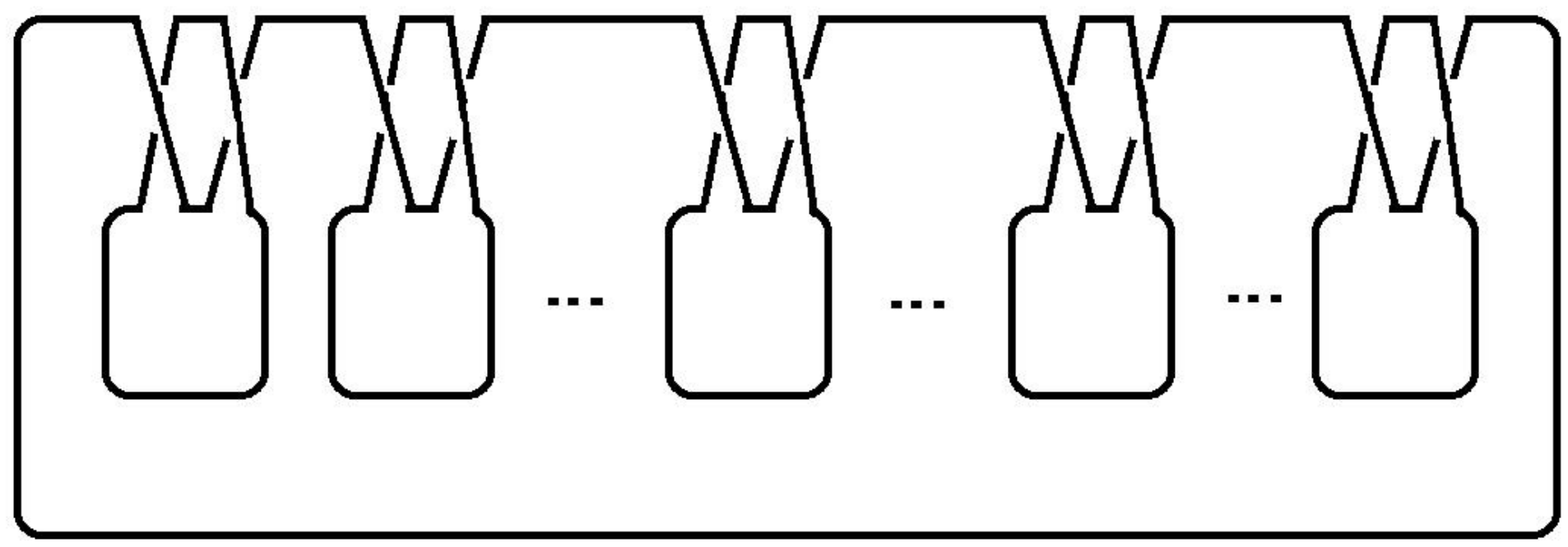}}
\end{picture}
\end{flushleft}
\caption{$e: \Sigma \times I \to S^3$}
\label{figure_e}
\end{figure}

\begin{figure}
\begin{flushleft}
\begin{picture}(200,150)
\put(0,-170){\includegraphics[width=360pt]{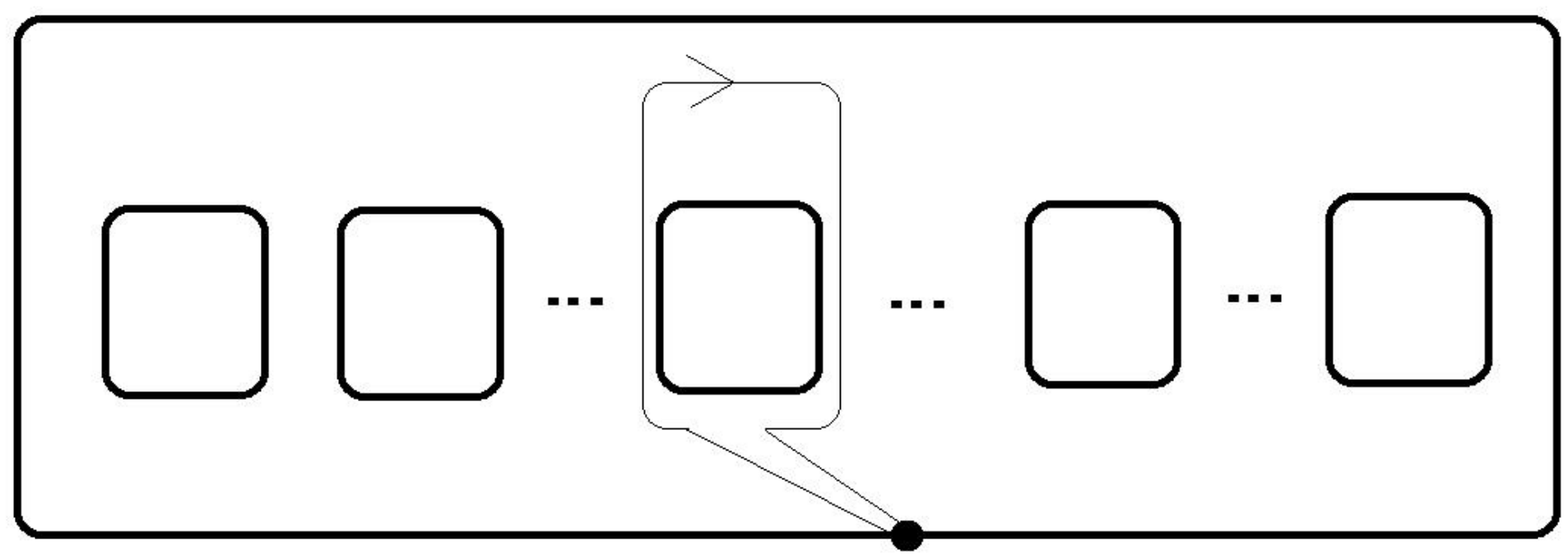}}
\put(45,30){$1$}
\put(160,30){$i$}
\put(160,92){$r_i$}
\put(300,30){$b$}
\end{picture}
\end{flushleft}
\caption{$r_{i}$}
\label{figure_r_i}
\end{figure}

We denote by $c$ an element represented by a tangle presented 
by simple closed curve $c
\in \shuugou{c_{i_1 i_2 \cdots i_j|1 \leq i_1<i_2< \cdots<i_j \leq b}}$,
$\kukakko{i,i} \defeq 2 (c_i -c_\emptyset)$ for $1 \leq i \leq b$, 
$\kukakko{i,j} \defeq
c_{ij}-c_{i}-c_{j}+c_\emptyset$ for $1 \leq i<j \leq b$ and 
$\kukakko{i,j,k} \defeq c_{ijk}-c_{ij}-c_{jk}-c_{ik}+
c_i+c_j+c_k-c_\emptyset$ for $1 \leq i <j<k \leq b$
where $c_\emptyset \defeq -A^2-A^{-2}$.
We need some claculations of $\vartheta$.

\begin{lemm}
\label{lemm_naiseki}
\begin{enumerate}
\item Let $i_1,i_2$  be elements of $\shuugou{1, \cdots, b}$.
We have
\begin{align*}
&\vartheta (\kukakko{i_1,i_1}\kukakko{i_2,i_2}) \\
&
\begin{cases}
=4(A^{12}+A^{8}+2A^7+2A^4+4A^3+3+2A^{-1}+A^{-4})
&\text{$i_1 =i_2$} \\
=4(A^3+1)^2(A^2+A^{-2})^2
&\text{$i_1 \neq i_2$},
\end{cases}
\end{align*}
and
\begin{align*}
&\vartheta (\kukakko{i_1,i_1}\kukakko{i_2,i_2}) =
\begin{cases}
240(A+1)^2 \mod ((A+1)^3)
&\text{$i_1 =i_2$} \\
144(A+1)^2 \mod ((A+1)^3)
&\text{$i_1 \neq i_2$}.
\end{cases}
\end{align*}
\item Let $i_1,j_1,i_2$ be elements of $\shuugou{1, \cdots, b}$ satisfying
$i_1 <j_1$. 
We have
\begin{align*}
&\vartheta (\kukakko{i_1,j_1}\kukakko{i_2,i_2}) \\
&
\begin{cases}
=-2(A^3+1)(A^{12}+A^{8}+2A^7+2A^4+4A^3+3+2A^{-1}+A^{-4})
&\text{$i_2 \in \shuugou{i_1,j_1}$} \\
=-2(A^3+1)^3(A^2+A^{-2})^2
&\text{$i_2 \notin \shuugou{i_1,j_1}$},
\end{cases}
\end{align*}
and
\begin{align*}
\vartheta (\kukakko{i_1,j_1}\kukakko{i_2,i_2}) =0 \mod ((A+1)^3).
\end{align*}
\item Let $i_1,j_1,k_1,i_2$ be elements of $\shuugou{1, \cdots, b}$ satisfying
$i_1 <j_1<k_1$. 
We have
\begin{align*}
&\vartheta (\kukakko{i_1,j_1,k_1}\kukakko{i_2,i_2}) \\
&
\begin{cases}
=2(A^3+1)^2(A^{12}+A^{8}+2A^7+2A^4+4A^3+3+2A^{-1}+A^{-4})
&\text{$i_2 \in \shuugou{i_1,j_1,k_1}$} \\
=2(A^3+1)^4(A^2+A^{-2})^2
&\text{$i_2 \notin \shuugou{i_1,j_1,k_1}$},
\end{cases}
\end{align*}
and
\begin{align*}
\vartheta (\kukakko{i_1,j_1,k_1}\kukakko{i_2,i_2}) =
0 \mod ((A+1)^4).
\end{align*}
\item
Let $i_1,j_1,i_2,j_2$ be elements of $\shuugou{1, \cdots,b}$
satisfying $i_1<j_1$ and $i_2 <j_2$.
We have
\begin{align*}
&\vartheta (\kukakko{i_1,j_1}\kukakko{i_2,j_2}) \\
&
\begin{cases}
=(A^3+1)^4(A^2+A^{-2})^2
&\text{$\sharp (\shuugou{i_1,j_1} \cap \shuugou{i_2,j_2}) =0$} \\
=(A^3+1)^2(A^{12}+A^{8}+2A^7+2A^4+4A^3+3+2A^{-1}+A^{-4})
&\text{$\sharp (\shuugou{i_1,j_1} \cap \shuugou{i_2,j_2}) =1$} \\
=A^{20}+A^{16}+4A^{15}+3A^{12}+4A^{11}+4A^{10}+2A^8+8A^7  \\
\ \ +8A^6+3A^4+12A^3+4A^2+5+4A^{-1}+A^{-4}
&\text{$\sharp (\shuugou{i_1,j_1} \cap \shuugou{i_2,j_2}) =2$},
\end{cases}
\end{align*}
and
\begin{align*}
&\vartheta (\kukakko{i_1,i_1}\kukakko{i_2,i_2}) =
\begin{cases}
0 \mod ((A+1)^3)
&\text{$\sharp (\shuugou{i_1,j_1} \cap \shuugou{i_2,j_2}) =0,1$} \\
48(A+1)^2 \mod ((A+1)^3)
&\text{$\sharp (\shuugou{i_1,j_1} \cap \shuugou{i_2,j_2}) =2$}.
\end{cases}
\end{align*}
\item
Let $i_1,j_1,k_1,i_2,j_2$ be elements of $\shuugou{1, \cdots,b}$
satisfying $i_1<j_1<k_1$ and $i_2 <j_2$.
We have
\begin{align*}
&\vartheta (\kukakko{i_1,j_1,k_1}\kukakko{i_2,j_2}) \\
&
\begin{cases}
=-(A^3+1)^5(A^2+A^{-2})^2
&\text{$\sharp (\shuugou{i_1,j_1,k_1} \cap \shuugou{i_2,j_2}) =0$} \\
=-(A^3+1)^3(A^{12}+A^{8}+2A^7+2A^4+4A^3+3+2A^{-1}+A^{-4})
&\text{$\sharp (\shuugou{i_1,j_1,k_1} \cap \shuugou{i_2,j_2}) =1$} \\
=-(A^3+1)(A^{20}+A^{16}+4A^{15}+3A^{12}+4A^{11}+4A^{10}+2A^8+8A^7  \\
\ \ +8A^6+3A^4+12A^3+4A^2+5+4A^{-1}+A^{-4})
&\text{$\sharp (\shuugou{i_1,j_1,k_1} \cap \shuugou{i_2,j_2}) =2$}, \\
\end{cases}
\end{align*}
and
\begin{align*}
&\vartheta (\kukakko{i_1,j_1,k_1}\kukakko{i_2,j_2}) =
0 \mod ((A+1)^3)
\end{align*}
\item
Let $i_1,j_1,k_1,i_2,j_2,k_2$ be elements of $\shuugou{1, \cdots,b}$
satisfying $i_1<j_1<k_1$ and $i_2 <j_2<k_2$.
We have
\begin{align*}
&\vartheta (\kukakko{i_1,j_1,k_1}\kukakko{i_2,j_2,k_2}) \\
&
\begin{cases}
=(A^3+1)^6(A^2+A^{-2})^2
&\text{$\sharp (\shuugou{i_1,j_1,k_1} \cap \shuugou{i_2,j_2,k_2}) =0$} \\
=(A^3+1)^4(A^{12}+A^{8}+2A^7+2A^4+4A^3+3+2A^{-1}+A^{-4})
&\text{$\sharp (\shuugou{i_1,j_1,k_1} \cap \shuugou{i_2,j_2,k_2}) =1$} \\
=(A^3+1)^2(A^{20}+A^{16}+4A^{15}+3A^{12}+4A^{11}+4A^{10}+2A^8+8A^7 \\
\ \ +8A^6+3A^4+12A^3+4A^2+5+4A^{-1}+A^{-4})
&\text{$\sharp (\shuugou{i_1,j_1,k_1} \cap \shuugou{i_2,j_2,k_2}) =2$} \\
=A^{28}+A^{24}+6A^{23}+4A^{20}+6A^{19}+12A^{18}+3A^{16}+18A^{15} \\
\ \ +12A^{14}+8A^{13}+6A^{12}+12A^{11}+24A^{10}+16A^9 +3A^8+18A^7\\
\ \ +36A^6+8A^5+4A^4
+30A^3+12A^2+9+6A^{-1}+A^{-4}
&\text{$\sharp (\shuugou{i_1,j_1,k_1} \cap \shuugou{i_2,j_2,k_2}) =3$},
\end{cases}
\end{align*}
and
\begin{align*}
&\vartheta (\kukakko{i_1,i_1}\kukakko{i_2,i_2}) =
\begin{cases}
0 \mod ((A+1)^4)
&\text{$\sharp (\shuugou{i_1,j_1,k_1} \cap \shuugou{i_2,j_2,k_2})=0,1,2$} \\
192(A+1)^3 \mod ((A+1)^4)
&\text{$\sharp (\shuugou{i_1,j_1,k_1} \cap \shuugou{i_2,j_2,k_2}) =3$}.
\end{cases}
\end{align*}
\end{enumerate}
\end{lemm}

\begin{proof}
Let $L_n$ be the link in $S^3$ as in Figure \ref{figure_L_n}.
We remark that $\mathcal{K} (L_n) =A^{2n}(A^4+1+A^{-4})+A^{-6n}$.
Let $i_1,i_2, \cdots,i_n$, $j_1,j_2, \cdots,j_m$ be elements of $\shuugou{1, \cdots,b}$
satisfying $i_1<i_2< \cdots < i_n$, $j_1<j_2< \cdots < j_m$
and $\sharp (\shuugou{i_1, \cdots,i_n} \cap \shuugou{j_1, \cdots, j_m}) =k$.
We have $\vartheta (c_{i_1 \cdots i_n } c_{j_1 \cdots j_m})
=(-A^3)^{n+m-k} \mathcal{K} (L_k)$.
Using this formula, we obtain the above equations.
This finishes the proof.

\begin{figure}
\begin{picture}(200,150)
\put(0,0){\includegraphics[width=200pt]{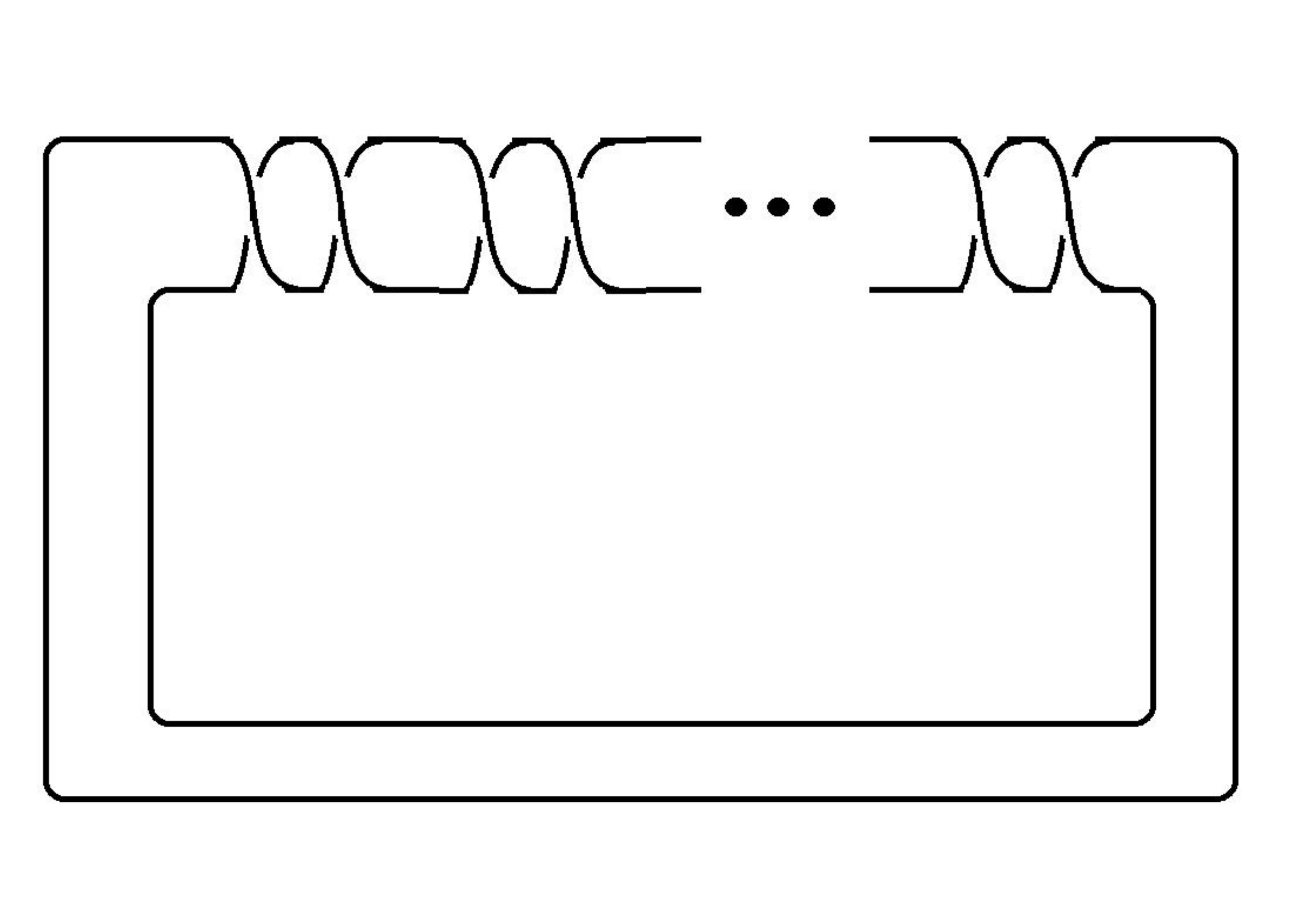}}
\put(45,87){$1$}
\put(78,87){$2$}
\put(155,87){$n$}
\end{picture}
\caption{$L(n)$}
\label{figure_L_n}
\end{figure}

\end{proof}

By Theorem \ref{thm_independent}, Theorem
\ref{thm_basis} follows from the following lemma.

\begin{lemm}
\label{lemm_basis}
The set 
\begin{align*}
\shuugou{1} \cup \shuugou{A+1} \cup \shuugou{\kukakko{i,j}|i \leq j} \cup
\shuugou{\kukakko{i,j,k}|i < j < k}
\end{align*}
is a basis of $\skein{\Sigma}/(\ker \epsilon)^2$
as a $\Q$ vector space.
\end{lemm}

\begin{proof}
It is enough to show
that if $X =Q+q(A+1) +\sum_{i}q_i(\kukakko{i,i}-12(A+1))+\sum_{i<j} q_{ij}\kukakko{i,j}
+\sum_{i<j<k}\kukakko{i,j,k} \in (\ker \epsilon)^2 $,
we have $Q=q=0$, $q_i=0$ for $1 \leq i \leq b$,
$q_{ij} =0$ for $1 \leq i<j \leq b$ and $q_{ijk} =0$ for $1 \leq i<j<k \leq b$.
We assume $X \in (\ker \epsilon)^2$. Since $X \in \ker \epsilon$, we have $Q=0$.
By Lemma \ref{lemm_naiseki}(1) (2) (3) (4) (5), we have
\begin{align*}
\vartheta (X^2) =(q^2 +\sum_i 96{q_i}^2+\sum_{i<j} 48{q_{ij}}^2)(A+1)^2 \mod ((A+1)^3).
\end{align*}
Since $\vartheta (X^2) \in ((A+1)^4)$, we obtain $
q=0$, $q_i =0$ for any $i$ and $q_{ij}=0$ for any $i$, $j$.
By Lemma \ref{lemm_naiseki} (6), we have
\begin{align*}
\vartheta (X^2) =(\sum_{i<j<k} 192{q_{ijk}}^2)(A+1)^3 \mod ((A+1)^4).
\end{align*}
Since $\vartheta (X^2) \in ((A+1)^4)$, we obtain 
$q_{ijk}=0$ for any $i$, $j$, $k$.
This proves the lemma.

\end{proof}

As a corollary of Theorem \ref{thm_basis}, we have the following.
\begin{cor}
If $\partial \Sigma \neq \emptyset$,
then $\Q$-linear map $\lambda: H_1 \wedge H_1 \wedge H_1 \to 
\ker \epsilon / (\ker \epsilon)^2$ is injective.
\end{cor}

\begin{proof}
We have $\shuugou{\kukakko{x_i,x_j,x_k}|i<j<k} =
\shuugou{\lambda ([x_i]\wedge [x_j] \wedge [x_k])|i<j<k}$.
This proves the corollary.
\end{proof}

\begin{cor}[\cite{TsujiCSAI} Remark 3.13]
Let $\Sigma$ be a compact connected surface of genus $g$ with $b+1$ boundary 
components. We assume $0 \leq b$.
We have $\dim_\Q (\ker \epsilon)/(\ker \epsilon)^2  =1+\frac{N(N+1)}{2} 
+\frac{N(N-1)(N-2)}{6}$, where $N=b+2g$. Furthermore $\dim_\Q (\ker \epsilon)^m/
(\ker \epsilon)^{m+1} < \infty$ for any $m$.
\end{cor}

\section{Filtrations}
\label{section_filtration}
In this section, Let $\Sigma$ be a connected compact oriented surface with non-empty
boundary.
In this section, we introduce a new filtration of the Kauffman bracket skein modules
of $\Sigma$.
We denote by $\varpi$ the natural quotient map $\ker \epsilon \to (\ker \epsilon /
(\ker / \epsilon)^2)/ \im \lambda$.

\begin{df}
The filtration $\filtn{F^n \skein{\Sigma}}$ of $\skein{\Sigma}$ is defined by 
\begin{align*}
&F^0  \defeq \skein{\Sigma}, \\
&F^1 \skein{\Sigma} = F^2 \skein{\Sigma} \defeq \ker \epsilon, \\
&F^3 \skein{\Sigma} \defeq  \ker \varpi, \\
&F^n \skein{\Sigma} \defeq \ker \epsilon F^{n-2} \skein{\Sigma} \ \ (\mathrm{for} \ \ 4 \leq n).
\end{align*}

\end{df}

We remark that $F^n \skein{\Sigma}$ is an ideal of $\skein{\Sigma}$ for any $n$.
By definition, The filtrations $\filtn{ (\ker \epsilon)^n}$ and $\filtn{F^n 
\skein{\Sigma}}$
induce the same topology of $\skein{\Sigma}$.
More precisely, we have  $F^{2n} \skein{\Sigma}
 =(\ker \epsilon)^n$.

\subsection{The filtrations depend only on the underlying 3-manifold}

We define another filtration $\filtn{F^{\star n} \skein{\Sigma}}$ of $\skein{\Sigma}$
which depends only on the underlying 3-manifold.

Let $\Sigma_{0,b+1}$ be a connected compact oriented surface of genus $0$
with $b+1$ boundary components.
We denote $\eta \defeq c_{123}-c_{12}-c_{23}-c_{13}+c_1+c_2+c_3+A^2+A^{-2} \in
\skein{\Sigma_{0,4}}$ and $\nu \defeq c_1+A^2+A^{-2}$.
Any embedding $\iota : \Sigma_{0,2} \times I \times
 \shuugou{1,2, \cdots, n} \coprod \Sigma_{0,4} \times I \times 
\shuugou{1,2, \cdots,m} \to \Sigma \times I $ induces $\iota_* :
(\oplus^n \skein{\Sigma_{0,2}}) \oplus (\oplus^m \skein{\Sigma_{0,4}}) \to
\skein{\Sigma}$.

\begin{df}
The filtration $\filtn{F^{\star n} \skein{\Sigma}}$ is defined as follows.

\begin{itemize}
\item $F^{\star 0} \skein{\Sigma} \defeq \skein{\Sigma}$ and 
$F^{\star 1} \skein{\Sigma} \defeq F^{\star 2} \skein{\Sigma}$.
\item If $n \geq 1$,
$F^{\star 2n} \skein{\Sigma} $ is the $\Q[A,\gyaku{A}]$-submodule
of $\skein{\Sigma}$ generated by
\begin{equation*}
\shuugou{\iota_*((\oplus^n \nu) \oplus (\oplus^{N-n} c_1))|
\iota:\Sigma_{0,2} \times I \times
 \shuugou{1,2, \cdots, N} \to \Sigma \times I \ \ \mathrm{embedding}}
\end{equation*}
and 
\begin{equation*}
(A+1)F^{\star 2n-2} \skein{\Sigma}.
\end{equation*}
\item If $n \geq 1$,
$F^{\star 2n+1} \skein{\Sigma} $ is the $\Q[A,\gyaku{A}]$-submodule
of $\skein{\Sigma}$ generated by
\begin{align*}
&\shuugou{\iota_*((\oplus^n \nu) \oplus (\oplus^{N-n} c_1) \oplus \eta) |
\iota:\Sigma_{0,2} \times I \times
 \shuugou{1,2, \cdots, N}   \coprod\Sigma_{0,4} \times I \to \Sigma \times I \ \ \mathrm{embedding}}, \\
&\shuugou{\iota_*((\oplus^{n+1} \nu) \oplus (\oplus^{N-n-1} c_1))|
\iota:\Sigma_{0,2} \times I \times
 \shuugou{1,2, \cdots, N} \to \Sigma \times I \ \ \mathrm{embedding}}
\end{align*}
and 
\begin{equation*}
(A+1)F^{\star 2n-1} \skein{\Sigma}.
\end{equation*}
\end{itemize}

\end{df}

\begin{lemm}
\label{lemm_filt_induction_1}
\begin{enumerate}
\item For an embedding $\iota :\Sigma_{0,4} \times I \to \Sigma \times I$,
we have $\iota_* (\eta)  \in \ker \varpi = F^3 \skein{\Sigma}$.
\item We have $F^3 \skein{\Sigma} = F^{\star 3} \skein{\Sigma} $.
\end{enumerate}
\end{lemm}

\begin{proof}
We choose $\gamma_1$, $\gamma_2$ and $ \gamma_3 \in \pi_1 (\Sigma)$ 
satisfying $\zettaiti{p_1 (\iota (r_i))} =\zettaiti{\gamma_i}$ for $i \in \shuugou{1,2,3}$.
We remark that
\begin{equation*}
w(\iota(c_{123}))-w(\iota (c_{12}))-w(\iota (c_{13}))-w(\iota (c_{23}))
+w (\iota (c_1))+w (\iota (c_2))+w (\iota (c_3))=0.
\end{equation*}
Here we also denote by $c \in \mathcal{T}(\Sigma_{0,4})$ an knot 
presented by $c$ for $c \in \shuugou{c_1,c_2,c_3,c_{12},c_{13},c_{23},c_{123}}$.
Hence we have
\begin{equation*}
\iota_* (\eta) =\kappa ((\gamma_1-1)(\gamma_2-1)(\gamma_3-1))
\mod (\ker \epsilon)^2.
\end{equation*}
This proves the lemma (1).

For any element $x \in F^3 \skein{\Sigma}$, we prove $x \in F^{\star 3} \skein{\Sigma}$.
We choose $ \gamma_1$, $\gamma_2$ and $\gamma_3  \in \pi_1 (\Sigma)$ 
and an embedding $\iota :\Sigma_{0,4} \times I \to \Sigma \times I$
satisfying
$x = \kappa ((\gamma_1-1)(\gamma_2-1)(\gamma_3-1)) \mod (\ker \epsilon)^2,$
and $\zettaiti{p_1 (\iota (r_i))} =\zettaiti{\gamma_i}$ for $i \in \shuugou{1,2,3}$.
Since $x =\iota_* (\eta) =0 \mod F^{\star3} \skein{\Sigma} \subset 
(\ker \epsilon)^2$, we have $ x \in F^{\star 3}\skein{\Sigma}$.

In order to prove $F^{\star 3} \skein{\Sigma} \subset
F^3 \skein{\Sigma}$,
it is enough to prove $\iota_* ((\oplus^N c_1) \oplus \eta) \in F^3 \skein{\Sigma}$ 
for any embedding $\iota:\Sigma_{0,2} \times I \times
 \shuugou{1,2, \cdots, N}   \coprod
\Sigma_{0,4} \times I \to \Sigma \times I$.
We have $\iota_* ((\oplus^N c_1) \oplus \eta) =(-2)^N \iota_* (\eta) \mod
F^{\star 3} \skein{\Sigma}$.
By this lemma (1), $\iota_*(\eta) \in F^3 \skein{\Sigma}$.

Hence we obtaion $F^{\star 3} \skein{\Sigma} =
F^3 \skein{\Sigma}$.
This proves the lemma (2).
\end{proof}

Let $c_1,c_2,c_3,c_{12},c_{23},c_{13},c_{123}$ and $c_\star$ be simple closed curves
in $\Sigma_{1,3}$ in Figure \ref{figure_sigma_1_3_1}, \ref{figure_sigma_1_3_2} and
\ref{figure_sigma_1_3_3}. We aloso denote by $c$ an element of $\skein{\Sigma_{1,3}}$
represented a knot presented by $c$ for $c \in \shuugou{c_1,c_2,c_3,c_{12},c_{23},c_{13},c_{123},c_\star}$.
We obtain the lemma by a straightforward calculation.

\begin{figure}
\begin{flushleft}
\begin{picture}(200,150)
\put(0,-170){\includegraphics[width=360pt]{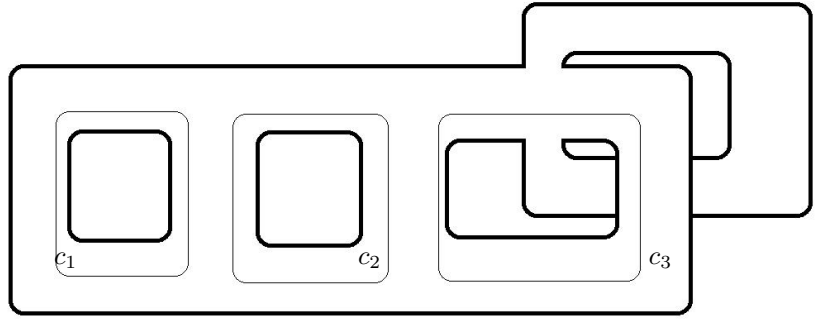}}
\put(45,23){$c_1$}
\put(160,23){$c_2$}
\put(270,23){$c_3$}
\end{picture}
\end{flushleft}
\caption{$c_1,c_2,c_3$}
\label{figure_sigma_1_3_1}
\end{figure}

\begin{figure}
\begin{flushleft}
\begin{picture}(200,150)
\put(0,-170){\includegraphics[width=360pt]{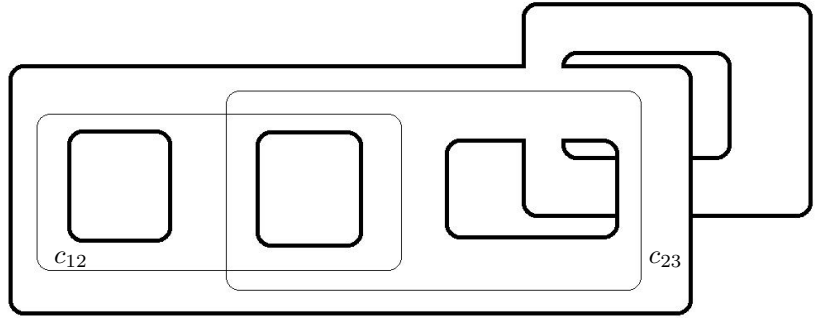}}
\put(45,23){$c_{12}$}
\put(270,23){$c_{23}$}
\end{picture}
\end{flushleft}
\caption{$c_{12},c_{23}$}
\label{figure_sigma_1_3_2}
\end{figure}

\begin{figure}
\begin{flushleft}
\begin{picture}(200,150)
\put(0,-170){\includegraphics[width=360pt]{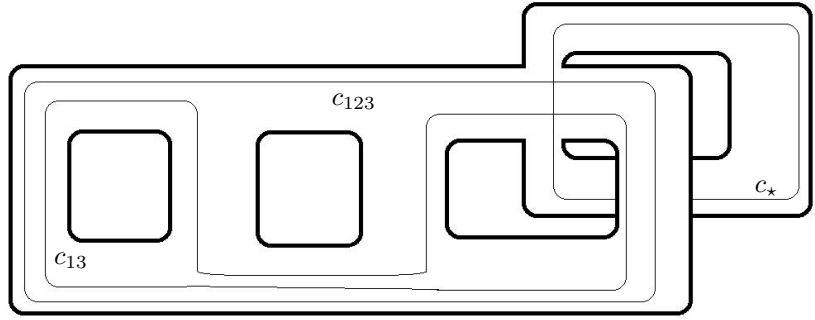}}
\put(45,23){$c_{13}$}
\put(310,50){$c_\star$}
\put(150,83){$c_{123}$}
\end{picture}
\end{flushleft}
\caption{$c_\star,c_{13},c_{123}$}
\label{figure_sigma_1_3_3}
\end{figure}

\begin{lemm}
\label{lemm_filt_lemm_1}
We have $\kukakko{1,2,3} c_\star- c_\star \kukakko{1,2,3} =
(-A+\gyaku{A}) (t_{c_\star}(\kukakko{1,2,3})-\gyaku{t_{c_\star}}(\kukakko{1,2,3}))$
where $\kukakko{1,2,3} =c_{123}-c_{12}-c_{23}-c_{13}+c_1+c_2+c_3+A^2+A^{-2}$
and $t_{c_\star}$ is a Dehn twist along $c_{\star}$.
\end{lemm}

Using this lemma, we have the following.

\begin{lemm}
\label{lemm_filt_douti}
We have $F^n \skein{\Sigma} =F^{\star n} \skein{\Sigma}$.
\end{lemm}

\begin{proof}
By \cite{TsujiCSAI} Lemma 5.3., we have 
$F^{2n} \skein{\Sigma} = F^{\star 2n } \skein{\Sigma} = (\ker \epsilon)^n$
for any $n$.
It is enough to show $F^{2n+1} \skein{\Sigma} = F^{\star 2n+1} \skein{\Sigma}$
for $n \in \Z_{\geq 1}$.
To prove it, we use the induction on $n$.
If $n=0$, the claim follows from $F^{\star 1} \skein{\Sigma} =F^{\star 2} \skein{\Sigma}
= \ker \epsilon =F^2 \skein{\Sigma} =F^1 \skein{\Sigma}$.
If $n=1$, the claim follows from Lemma \ref{lemm_filt_induction_1} (2).
We assume $F^{2n+1}\skein{\Sigma} = F^{\star 2n+1} \skein{\Sigma}$.
The embeddings $\iota, \iota':\Sigma_{0,2} \times I \times
 \shuugou{1,2, \cdots, N}   \coprod\Sigma_{0,4} \times I \to \Sigma \times I $
are only differ in an open ball in $\Sigma \times I$ shown in
Figure \ref{figure_iota} and Figure
\ref{figure_iota'}, respectively.
By Lemma \ref{lemm_filt_lemm_1}, we have
\begin{align*}
\iota_*((\oplus^{n+1} \nu) \oplus (\oplus^{N-n-1} c_1) \oplus \eta)-
\iota'_*((\oplus^{n+1} \nu) \oplus (\oplus^{N-n-1} c_1) \oplus \eta) \\
\in (-A+A^{-1}) F^{\star 2n+1} \skein{\Sigma} =(A+1) F^{2n+1} \skein{\Sigma}.
\end{align*}
Using this equation repeatedly, we have
\begin{align*}
&F^{\star 2n+3} \skein{\Sigma} \\
&=(-A+\gyaku{A})F^{\star 2n+1} \skein{\Sigma}
+F^{\star 2n}\skein{\Sigma} F^{\star 3}\skein{\Sigma} \\
&=(-A+\gyaku{A})F^{2n+1} \skein{\Sigma}
+F^{2n}\skein{\Sigma} F^{3}\skein{\Sigma} \\
&=F^{2n+3}\skein{\Sigma}.
\end{align*}

This proves the lemma.

\begin{figure}
\begin{picture}(200,150)
\put(-80,0){\includegraphics[width=300pt]{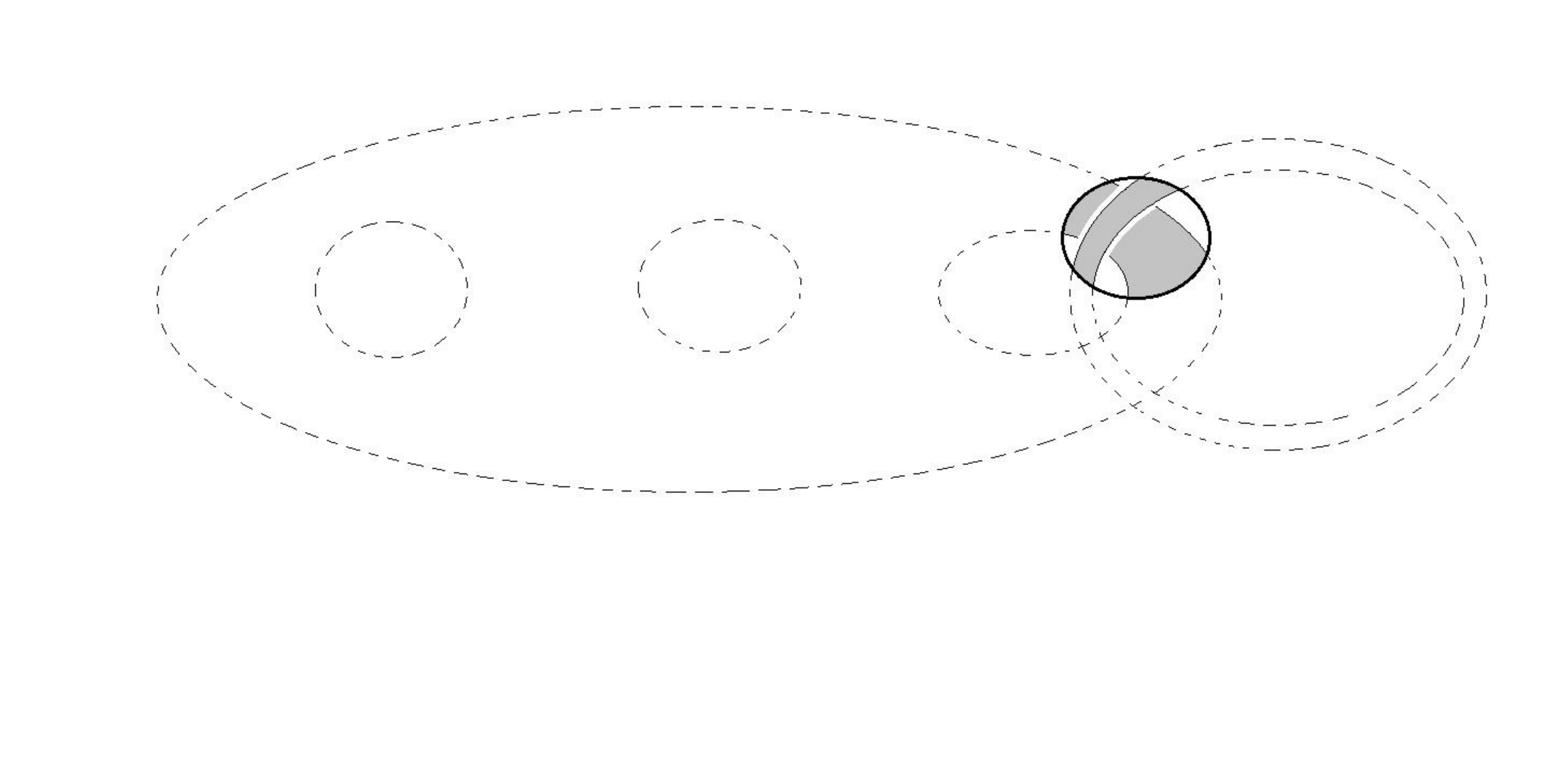}}
\end{picture}
\caption{$\iota$}
\label{figure_iota}
\end{figure}

\begin{figure}
\begin{picture}(200,150)
\put(-80,0){\includegraphics[width=300pt]{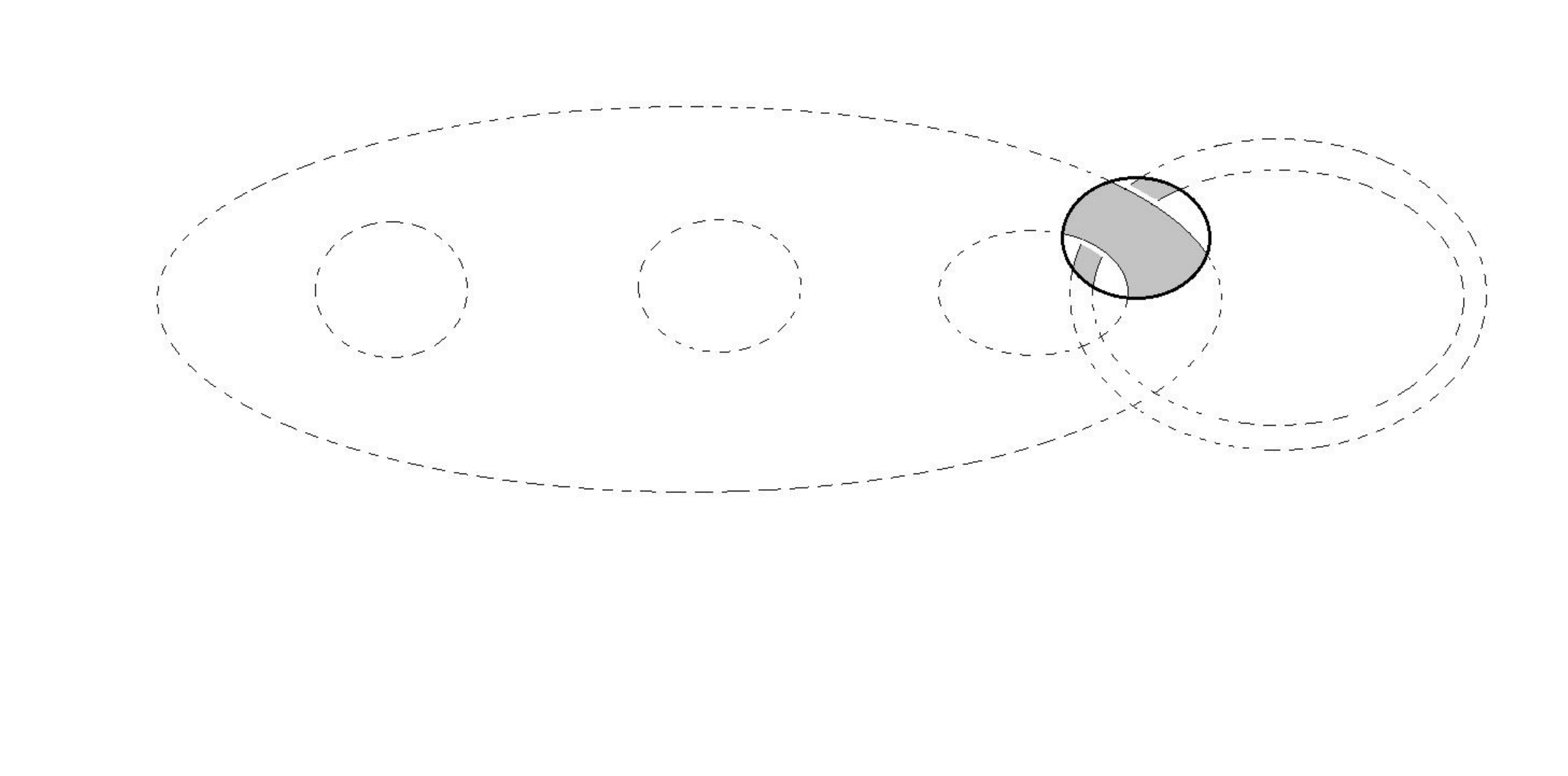}}
\end{picture}
\caption{$\iota'$}
\label{figure_iota'}
\end{figure}

\end{proof}

Using this lemma, we have the following.

\begin{thm}
\label{thm_filtration_independent}
Let $\Sigma$ and $\Sigma'$ be two oriented compact connected surfaces
such that 
there exists a diffeomorphism $\mathcal{X} : (\Sigma \times I)
 \to (\Sigma' \times I)$. 
Then we have $\mathcal{X} (F^n \skein{\Sigma}) = F^n \skein{\Sigma' }$
for $n \geq 0$.
\end{thm}

\begin{proof}
By definition, we have $\mathcal{X} (F^{\star n} \skein{\Sigma})=
F^{\star n} \skein{\Sigma'}$.
By Lemma \ref{lemm_filt_douti}, we have
$F^{\star n} \skein{\Sigma} =F^{ n} \skein{\Sigma}$ and
$F^{\star n} \skein{\Sigma'} =F^{ n} \skein{\Sigma'}$.
Hence we have $\mathcal{X} (F^{ n} \skein{\Sigma})=
F^{n} \skein{\Sigma'}$.
This proves the theorem.

\end{proof}

\begin{cor}
\label{cor_filtration_bilinear}
We have $F^{\star n} \skein{D^2} =(\ker \epsilon)^{\gauss{\frac{n+1}{2}}}=
(A+1)^{\gauss{\frac{n+1}{2}}} \Q[\pm A]$.
Furthermore, in the situation of section \ref{section_basis},
we have $\theta_\xi (F^n \skein{\Sigma}) \subset (A+1)^{\gauss{\frac{n+1}{2}}} \Q[\pm A]$.
\end{cor}


\subsection{The product and the filtration}
In this subsection, we prove the following.

\begin{prop}
\label{prop_product_filtration}
For $n$ and $m \in \Z_{\geq 0}$, we have
$F^n \skein{\Sigma} F^m \skein{\Sigma} 
\subset F^{n+m} \skein{\Sigma}$.
\end{prop}

For $N \in \Z_{\geq 2}$ and $1 \leq i_1<i_2< \cdots <i_j \leq N$,
we denote
\begin{equation*}
\kukakko{i_1,i_2, \cdots,i_j} \defeq
\sum_{\shuugou{k_1,\cdots,k_l} \subset \shuugou{i_1, \cdots,i_j},
k_1 <k_2< \cdots <k_l}(-1)^{j-l}c_{k_1 \cdots k_l} \in \skein{\Sigma_{0,N+1}}.
\end{equation*}
Here we define $c_{\emptyset} \defeq -A^2 -A^{-2}$.

\begin{lemm}
\label{lemm_product_filtration}
We have $\kukakko{1,2,4}\kukakko{3,5,6}  \in (\ker \epsilon)^3  =F^6 \skein{\Sigma_{0,7}}$.
\end{lemm}

\begin{proof}
By a straight calculation, we have
\begin{align*}
&\kukakko{1,3}\kukakko{2,4}+\kukakko{2,4}\kukakko{1,3}= \\ 
&(A^2+A^{-2})\kukakko{1,2,3,4} 
+(A^4+A^{-4})(\kukakko{1,2}\kukakko{3,4} +\kukakko{1,4}\kukakko{2,3})
\\
&+(A^2+A^{-2})(\kukakko{1,2}\kukakko{3}\kukakko{4}+
\kukakko{2,3}\kukakko{4}\kukakko{1}+
\kukakko{1,4}\kukakko{2}\kukakko{3}
+\kukakko{3,4}\kukakko{1}\kukakko{2}) \\
&+2(\kukakko{1}\kukakko{2,3,4}+\kukakko{2}\kukakko{1,3,4}
+\kukakko{3}\kukakko{1,2,4}+\kukakko{4}\kukakko{1,2,3}) \\
&+2\kukakko{1}\kukakko{2}\kukakko{3}\kukakko{4},
\end{align*}
and
\begin{align*}
&\kukakko{1,3}\kukakko{2,4}-\kukakko{2,4}\kukakko{1,3} \\
&=(A^2-A^{-2}) \\
&(\frac{1}{2}(2\kukakko{1,4}-\kukakko{1}\kukakko{4}
+(A-\gyaku{A})^2(\kukakko{1}+\kukakko{4}))
(2\kukakko{2,3}-\kukakko{2}\kukakko{3}
+(A-\gyaku{A})^2(\kukakko{2}+\kukakko{3})) \\
&-\frac{1}{2}(2\kukakko{1,2}-\kukakko{1}\kukakko{2}
+(A-\gyaku{A})^2(\kukakko{1}+\kukakko{2}))
(2\kukakko{3,4}-\kukakko{3}\kukakko{4}
+(A-\gyaku{A})^2(\kukakko{3}+\kukakko{4})) \\
&+(A-\gyaku{A})^2(c_{12}c_{34}-c_{14}c_{23})).
\end{align*}

For $y_1,y_2,y_3,y_4 \in \pi_1 (\Sigma_{0,7})$,
we denote by $\iota(y_1,y_2,y_3,y_4)$ the embedding
$\Sigma_{0,4} \to \Sigma_{0,7}$ inducing 
 $\pi_1(\Sigma_{0,4}) \to \pi_1 (\Sigma_{0,7})$
such that $\gamma_i \mapsto y_i$ for $i=1,2,3,4$.
We aslso denote by $\iota(y_1,y_2,y_3,y_4)$
the homeomorphism $\widehat{\skein{\Sigma_{0,4}}}
\to \widehat{\skein{\Sigma_{0,7}}}$ induced by
the embedding $\iota(y_1,y_2,y_3,y_4)$.

Using the equations staede the above and the two $\Q[A,\gyaku{A}]$-module homomorphisms
\begin{align*}
&\sum_{\epsilon_1, \epsilon_2, \epsilon_3 \in \shuugou{0,1}}
(-1)^{\epsilon_1+\epsilon_2+\epsilon_3} \iota (\gamma_1^{\epsilon_1}
\gamma_2^{\epsilon_2}\gamma_3^{\epsilon_3}, \gamma_4,\gamma_5,\gamma_6) \\
&\sum_{\epsilon_1, \epsilon_2, \epsilon_3, \epsilon_4 \in \shuugou{0,1}}
(-1)^{\epsilon_1+\epsilon_2+\epsilon_3+\epsilon_4}
\iota(\gamma_1^{\epsilon_1} \gamma_2^{\epsilon_2}, \gamma_3,\gamma_4,
\gamma_5^{\epsilon_3}\gamma_6^{\epsilon_4}),
\end{align*}
we have
\begin{align*}
&(A^2+A^{-2})\kukakko{1,2,3,4,5,6} \\
&+(A^4+A^{-4})(\kukakko{1,2,3}\kukakko{4,5,6} +\kukakko{1,2,5,6}\kukakko{3,4
})
-\kukakko{1,2,4}\kukakko{3,5,6}-\kukakko{3,5,6}\kukakko{1,2,4} \\
&+(A^2+A^{-2})(\kukakko{1,2,3}\kukakko{4}\kukakko{5,6}+
\kukakko{3,4}\kukakko{5,6}\kukakko{1,2}+
\kukakko{1,2,5,6}\kukakko{3}\kukakko{4}
+\kukakko{1,2}\kukakko{3}\kukakko{4,5,6}) \\
&+2(\kukakko{1,2}\kukakko{3,4,5,6}+\kukakko{3}\kukakko{1,2,4,5,6}
+\kukakko{4}\kukakko{1,2,3,5,6}+\kukakko{5,6}\kukakko{1,2,3,4}) \\
&+2\kukakko{1,2}\kukakko{3}\kukakko{4}\kukakko{5,6} \\
&=(A^2+A^{-2})\kukakko{1,2,3,4,5,6} \\
&+(A^4+A^{-4})(\kukakko{1,2,3,4}\kukakko{5,6} +\kukakko{1,2,3,6}\kukakko{4,5})
-\kukakko{1,2,3,5,}\kukakko{4,6}-\kukakko{4,6}\kukakko{1,2,3,5} \\
&+(A^2+A^{-2})(\kukakko{1,2,3,4}\kukakko{5}\kukakko{6}+
\kukakko{4,5}\kukakko{6}\kukakko{1,2,3}+
\kukakko{1,2,3,6}\kukakko{4}\kukakko{5}
+\kukakko{4,5}\kukakko{6}\kukakko{1,2,3}) \\
&+2(\kukakko{1,2,3}\kukakko{4,5,6}+\kukakko{4}\kukakko{1,2,3,5,6}
+\kukakko{6}\kukakko{1,2,3,4,5}+\kukakko{5}\kukakko{1,2,3,4,6}) \\
&+2\kukakko{1,2,3}\kukakko{4}\kukakko{5}\kukakko{6} \\
&=0,
\end{align*} 
and
\begin{align*}
&\kukakko{1,2,4}\kukakko{3,5,6}-\kukakko{3,5,6}\kukakko{1,2,4} \\
&=(A^2-A^{-2}) \\
&(\frac{1}{2}(2\kukakko{1,2,5,6}-\kukakko{1,2}\kukakko{5,6}
+(A-\gyaku{A})^2(\kukakko{1,2}+\kukakko{5,6}))
(2\kukakko{3,4}-\kukakko{3}\kukakko{4}
+(A-\gyaku{A})^2(\kukakko{3}+\kukakko{4})) \\
&-\frac{1}{2}(2\kukakko{1,2,3}-\kukakko{1,2}\kukakko{3}
+(A-\gyaku{A})^2(\kukakko{1,2}+\kukakko{3}))
(2\kukakko{4,5,6}-\kukakko{4}\kukakko{5,6}
+(A-\gyaku{A})^2(\kukakko{4}+\kukakko{5,6})) \\
&+(A-\gyaku{A})^2(
(c_{123}-c_{13}-c_{23})(c_{456}-c_{45}-c_{46})
- \\
&(c_{1256}-c_{125}-c_{126}-c_{156}-c_{256}+
c_{15}+c_{16}+c_{25}+c_{26})c_{34})).
\end{align*}
By these equations, 
we have $\kukakko{1,2,4}\kukakko{3,5,6}  \in (\ker \epsilon)^3  =F^6 \skein{\Sigma_{0,7}}$
and $\kukakko{1,2,4}\kukakko{3,5,6}-\kukakko{3,5,6}\kukakko{1,2,4} \in 
(\ker \epsilon)^4 =F^8 \skein{\Sigma_{0,7}}$.
This finishes the proof.
\end{proof}

\begin{cor}
\label{cor_calculation_sigma_0_5}
We have $\kukakko{1,2,3,4} \in (\ker \epsilon)^2 = F^4 \skein{\Sigma_{0,5}}$.
Furthermore, we have $\kukakko{1,2,3,4} =\frac{1}{2} (
-\kukakko{1,2}\kukakko{3,4}-\kukakko{1,4}\kukakko{2,3}
+\kukakko{1,3}\kukakko{2,4}) \mod F^5 \skein{\Sigma_{0,5}}$.
\end{cor}

\begin{cor}
We have $2\kukakko{1,2,4}\kukakko{3,5,6} =
\kukakko{1,3}\kukakko{2,5}\kukakko{4,6}
+\kukakko{1,5}\kukakko{2,6}\kukakko{3,4}
+\kukakko{1,6}\kukakko{2,3}\kukakko{4,5}
-\kukakko{1,3}\kukakko{2,6}\kukakko{4,5}
-\kukakko{2,5}\kukakko{1,6}\kukakko{3,4}
-\kukakko{1,5}\kukakko{2,3}\kukakko{4,6} \mod F^7 \skein{\Sigma_{0,7}}$.
\end{cor}

\begin{proof}
By the proof of Lemma \ref{lemm_product_filtration}, we have
\begin{align*}
& -\kukakko{1,2}\kukakko{5,6}\kukakko{3,4}
-\kukakko{1,6}\kukakko{2,5}\kukakko{3,4}
+\kukakko{1,5}\kukakko{2,6}\kukakko{3,4} \\
&-2 \kukakko{1,2,4}\kukakko{3,5,6}+2\kukakko{1,2}\kukakko{3,4}\kukakko{5,6} \\
& -\kukakko{1,2}\kukakko{3,4}\kukakko{5,6}
-\kukakko{1,2}\kukakko{3,6}\kukakko{4,5}
+\kukakko{1,2}\kukakko{3,5}\kukakko{4,6} \\
&= -\kukakko{1,2}\kukakko{3,6}\kukakko{4,5}
-\kukakko{1,6}\kukakko{2,3}\kukakko{4,5}
+\kukakko{1,3}\kukakko{2,6}\kukakko{4,5} \\
&+\kukakko{1,2}\kukakko{3,5}\kukakko{4,6}
+\kukakko{1,5}\kukakko{2,3}\kukakko{4,6}
-\kukakko{1,3}\kukakko{2,5}\kukakko{4,6} \mod F^7 \skein{\Sigma_{0,7}}.
\end{align*}
This proves the corollary.

\end{proof}

\begin{proof}[Proof of Proposition \ref{prop_product_filtration}]
By Lemma \ref{lemm_filt_douti}, we have
\begin{equation*}
F^2 \skein{\Sigma} F^3 \skein{\Sigma} =F^3 \skein{\Sigma} F^2 \skein{\Sigma}
=F^{\star 5 }\skein{\Sigma} =F^5 \skein{\Sigma}.
\end{equation*}
Using this equation, if one of $n, m \in \Z_{\geq 0}$ is even number,
we have
\begin{equation*}
F^n \skein{\Sigma} F^m \skein{\Sigma} \subset F^{n+m} \skein{\Sigma}.
\end{equation*}
It is enough to show 
\begin{equation*}
F^{2n+1} \skein{\Sigma} F^{2m+1} \skein{\Sigma}
\subset F^{2n+2m+2} \skein{\Sigma}.
\end{equation*}
We choose two embedding
\begin{align*}
&\iota_1 :\Sigma_{0,2} \times I\times \shuugou{1,2,\cdots,N} \coprod \Sigma_{0,4} \times I 
\to \Sigma \times [\frac{1}{2},1], \\
&\iota_2 :\Sigma_{0,2} \times I\times \shuugou{1,2,\cdots,M} \coprod \Sigma_{0,4} \times I 
\to \Sigma \times [0,\frac{1}{2}]. \\
\end{align*}

Since 
$
F^{\star 2n+2m+2} \skein{\Sigma} =F^{2n+2m+2} \skein{\Sigma}$,
it is enough to show
\begin{equation*}
\iota_1 (\oplus^{n-1} \nu \oplus \oplus^{N-n+1} c_1 \oplus \eta)
\iota_2(\oplus^{m-1} \nu \oplus  \oplus^{M-m+1} c_1 \oplus \eta)
\in F^{\star 2n+2m+2}.
\end{equation*}
Choose an embedding
\begin{equation*}
\iota:\Sigma_{0,2} \times \shuugou{1,2, \cdots ,N+M} \coprod
\Sigma_{0,7} \times I \to \Sigma \times I
\end{equation*}
satisfying 
\begin{align*}
&\iota_{|\Sigma_{0,2} \times I \times \shuugou{n}} 
=\iota_{1| \Sigma_{0.2} \times I \times \shuugou{n}} \\
&\iota_{|\Sigma_{0,2} \times I \times \shuugou{N+m}}
= \iota_{2|\Sigma_{0,2} \times I\times \shuugou{m}}
\end{align*}
for any $n \in \shuugou{1,2, \cdots, N}$
and $m \in \shuugou{1,2, \cdots,M}$ and
\begin{align*}
&\iota_1 (\oplus^{n-1} \nu \oplus \oplus^{N-n+1} c_1 \oplus \eta)
\iota_2(\oplus^{m-1} \nu \oplus  \oplus^{M-m+1} c_1 \oplus \eta) \\
&=
\iota ((\oplus^{n-1} \nu \oplus \oplus^{N-n+1} c_1) \oplus
(\oplus^{m-1} \nu \oplus  \oplus^{M-m+1} c_1) \oplus (\kukakko{1,2,4}
\kukakko{3,5,6})).
\end{align*}
By Lemma \ref{lemm_product_filtration},
we have
\begin{equation*}
\iota_1 (\oplus^{n-1} \nu \oplus \oplus^{N-n+1} c_1 \oplus \eta)
\iota_2(\oplus^{m-1} \nu \oplus  \oplus^{M-m+1} c_1 \oplus \eta)
\in F^{\star 2n+2m+2}
\end{equation*}
This proves the theorem.

\end{proof}

\subsection{The Goldman Lie algebra and its filtration}
We first review some classical facts about the Goldman Lie algebra of $\Sigma$
and the group ring of $\pi_1 (\Sigma)$. Fix a base point $*,*'$ of $\partial \Sigma$.
We denote by $\pi_1 (\Sigma, *)$ the fundamental group of $\Sigma$.
and by $\hat{\pi} (\Sigma)$ the set of conjugacy classes of  $\pi_1(\Sigma, *)$.
Furthermore, we denote by $\pi_1 (\Sigma, *, *')$ the fundamental 
groupoid from $*$ to $*'$.
Let $\zettaiti{\cdot}:\pi_1(\Sigma ,*) \to \hat{\pi}(\Sigma)$
be the quotient map.

We consider the action $\sigma_\pi :\Q \hat{\pi} (\Sigma) \times
\Q \pi (\Sigma,*,*') \to \Q \pi (\Sigma,*,*')$ defined by
\begin{equation*}
\sigma_\pi(\zettaiti{x})(r) \defeq 
\sum_{p \in x \cap r} \epsilon(p,x,r)
r_{*_1 p}x_p r_{p *_2}
\end{equation*}
for $\zettaiti{x} \in \hat{\pi} (\Sigma) $ and $r \in \pi (\Sigma,*_1, *_2)$ in general position.
For details, see \cite{Kawazumi} Definition 3.2.1.
We remark $[x, \zettaiti{r}] =\zettaiti{\sigma (x)(r)}$
 for $x \in \Q \hat{\pi}$ and $ y \in \Q \pi_1 (\Sigma ,*)$.

We denote by $\epsilon_\pi: \Q \pi_1(\Sigma, *) \to \Q$ the augmentation map
defined by $x \in \pi_1 (\Sigma, *) \mapsto 1$.

\begin{prop}[\cite{Kawazumi} Theorem 4.1.2]
\label{prop_Goldman_bracket}
We have
\begin{align*}
&\sigma (\zettaiti{(\ker \epsilon_\pi)^n}) 
((\ker \epsilon_\pi)^m (\Q \pi_1 (\Sigma, *, *'))) \subset
(\ker \epsilon_\pi)^{n+m-2} (\Q \pi_1 (\Sigma, *,*') \\
&[\zettaiti{(\ker \epsilon_\pi)^n}, \zettaiti{(\ker \epsilon_\pi)^m}]
\subset \zettaiti{(\ker \epsilon_\pi)^{n+m-2}}
\end{align*}
for any $n$ and $m$.
Furthermore, we have
\begin{align*}
&\sigma (\zettaiti{(\ker \epsilon_\pi)^n}) 
(\Q \pi_1 (\Sigma, *, *')) \subset
(\ker \epsilon_\pi)^{n-1} (\Q \pi_1 (\Sigma, *,*')
\end{align*}
for any $n$ and $m$.
\end{prop} 

We denote by $H_1 \defeq H_1 (\Sigma, \Q) H$.
Using  a Magnus expansion,
we have the following.
See, for example, \cite{Kawazumi05} \cite{KK}.

\begin{prop}
The following $\Q$ linear map 
\begin{align*}
& R(n) :(\ker \epsilon_\pi)^n \Q \pi_1 (\Sigma, *,*') /(\ker \epsilon_\pi)^{n+1}
\Q \pi_1 (\Sigma, *,*') \to H_1^{\otimes n},
(x_1-1) \cdots (x_n-1)r \mapsto [x_1] \otimes \cdots \otimes [x_n] \\
&C(n) :\zettaiti{(\ker \epsilon_\pi)^n}/\zettaiti{(\ker \epsilon_\pi)^{n+1}} \to c(H_1^{\otimes n}),
\zettaiti{(x_1-1) \cdots (x_n-1)} 
\mapsto c([x_1] \otimes \cdots \otimes [x_n]) \\
\end{align*}
are well defined and isomorphisms
where $c:H_1^{\otimes n} \to H_1^{\otimes n}$ is defined by
$c([x_1] \otimes \cdots \otimes [x_n]) = \sum_{i=1}^n [x_i] \otimes [x_{i+1}]
\otimes \cdots \otimes  [x_n] \otimes [x_1] \otimes [x_2] \otimes 
\cdots \otimes [x_{i-1}]$.
\end{prop}

Let $\mu :H_1 \times H_1 \to \Q$ be the intersection form.
By the above proposition,
the action $\sigma_\pi$ induces the action $\sigma_{\pi,n,m}:
\zettaiti{(\ker \epsilon_\pi)^n}/\zettaiti{(\ker \epsilon_\pi)^{n+1}} \times
(\ker \epsilon_\pi)^m/(\ker \epsilon_\pi)^{m+1} \to
(\ker \epsilon_\pi)^{n+m-2}/(\ker \epsilon_\pi)^{n+m-1}$
 and the bracket $[,]$ of $\Q \hat{\pi} (\Sigma)$
induces the bracket $b_{\pi,n,m}:
\zettaiti{(\ker \epsilon_\pi)^n}/\zettaiti{(\ker \epsilon_\pi)^{n+1}} \times
\zettaiti{(\ker \epsilon_\pi)^m}/\zettaiti{(\ker \epsilon_\pi)^{m+1}} \to
\zettaiti{(\ker \epsilon_\pi)^{n+m-2}}/\zettaiti{(\ker \epsilon_\pi)^{n+m-1}}$.
We denote $\sigma'_{\pi,n,m} \defeq R(n+m-2) \circ \sigma_{\pi,n,m} 
\circ (C(n)^{-1} \times R(m)^{-1})$ and $b'_{\pi,n,m} \defeq
C(n+m-2) \circ b_{\pi,n,m} \circ (C(n)^{-1} \times C(m)^{-1})$
\begin{prop}[\cite{Kawazumi05},\cite{KK}]
\label{prop_Goldman_Lie_filtration}
We have
\begin{align*}
&\sigma'_{\pi, n,m} (c(a_1 \otimes  \cdots \otimes a_n))
(b_1 \otimes \cdots \otimes b_m) \\
&=\sum_{i=1}^n \sum_{j=1}^m
\mu(a_i,b_j) b_1 \otimes \cdots \otimes b_{j-1} 
\otimes a_{i+1} \otimes  \cdots \otimes a_n \otimes a_1
\otimes \cdots \otimes a_{i-1} \otimes b_{j+1}
\otimes \cdots \otimes b_m \\
&\sigma'_{\pi, n,m} (c(a_1 \otimes  \cdots \otimes a_n))
(c(b_1 \otimes \cdots \otimes b_m)) \\
&=\sum_{i=1}^n \sum_{j=1}^m
\mu(a_i,b_j) c(b_1 \otimes \cdots \otimes b_{j-1} 
\otimes a_{i+1} \otimes  \cdots \otimes a_n \otimes a_1
\otimes \cdots \otimes a_{i-1} \otimes b_{j+1}
\otimes \cdots \otimes b_m).
\end{align*}
\end{prop}

\subsection{The bracket and the filtration}

In this subsection, we prove the proposition.

\begin{prop}
\label{prop_bracket_filtration}
We have $[F^n \skein \Sigma, F^m \skein \Sigma] \subset  F^{n+m-2} \skein{\Sigma}$.
\end{prop}

In order to prove this proposition,
by the Leibniz rule, it is enough to show Lemma \ref{lemm_bracket_filtration}.

By Theorem \ref{thm_kappa},
the equation (\ref{equation_calc_2}) and 
Proposition \ref{prop_Goldman_bracket} (\cite{KK} Theorem 4.1.2), we have the following.

\begin{prop}
\label{prop_bracket_filtration_2}
\begin{enumerate}
\item The Lie algebra homomorphism $\kappa :\zettaiti{\Q \pi}_\square \to
\ker \epsilon / (\ker \epsilon)^2$
induces 
$\kappa :\zettaiti{\Q \pi}_\square /\zettaiti{(\ker \epsilon_\pi)^3}_\square
\to \ker \epsilon/\ker \varpi $.
\item The Lie algebra homomorphism $\kappa :\zettaiti{\Q \pi}_\square \to
\ker \epsilon / (\ker \epsilon)^2$
induces 
$\kappa :\zettaiti{\Q \pi}_\square /\zettaiti{(\ker \epsilon_\pi)^4}_\square
\to \ker \epsilon/(\ker \epsilon)^2$.
\end{enumerate}•

\end{prop}

In order to prove Proposition \ref{prop_bracket_filtration},
we need the lemma.

\begin{lemm}
\label{lemm_bracket_filtration}
\begin{enumerate}
\item  We have
$[\ker \epsilon, \ker \epsilon ] \subset \ker \epsilon$.
\item We have 
$[\ker \epsilon, \ker \varpi] \subset \ker \varpi$.
\item We have
$[\ker \varpi,\ker \varpi] \subset (\ker \epsilon)^2$.
\end{enumerate}
\end{lemm}

\begin{proof}
The claim (1) is \cite{TsujiCSAI} Lemma 3.11.
By \cite{KK} Theorem 4.1.2, we have 
\begin{align*}
&[\ker \epsilon, \im \lambda ] =\kappa(
[\zettaiti{(\epsilon_\pi)^2}_\square,\zettaiti{(\epsilon_\pi)^3}_\square]
\subset \kappa(\zettaiti{(\epsilon_\pi)^3}_\square) =\im \lambda \mod (\ker \epsilon)^2, \\
&[\im \lambda, \im \lambda ] =\kappa(
[\zettaiti{(\epsilon_\pi)^3}_\square,\zettaiti{(\epsilon_\pi)^3}_\square]
\subset \kappa(\zettaiti{(\epsilon_\pi)^4}_\square) =0 \mod (\ker \epsilon)^2. \\
\end{align*}
This proves the lemma.
\end{proof}

\begin{df}
The $\Q$-linear map $\rho :\Q \oplus H_1 \cdot H_1 \to
F^2 \skein{\Sigma} /F^3 \skein{\Sigma}$ is defined by
$1 \in \Q \mapsto (A+1)$ and $[a] \cdot [b] \mapsto \kukakko{ab}-\kukakko{a}
-\kukakko{b}$ for $a$ and $b \in \pi_1 (\Sigma)$
where $H_1 \cdot H_1$ is the symmetric tensor of $H_1$.
\end{df}
 
By Theorem \ref{thm_basis}, $\rho$ is bijection.

By Proposition \ref{prop_bracket_filtration_2} and 
Proposition \ref{prop_Goldman_Lie_filtration}
we have following.

\begin{cor}
\label{cor_bracket_filtration}
\begin{enumerate}
\item We have 
\begin{equation*}
[\rho(\alpha_1 \cdot \alpha_2), \rho(\beta_1 \cdot \beta_2)]
=\sum_{(i_1,i_2) =(1,2),(2,1),(j_1,j_2)=(1,2),(2,1)}
 -2 \mu(\alpha_{i_1},\beta_{j_1})\rho(\alpha_{i_2}\cdot \beta_{j_2})
\mod F^3 \skein{\Sigma}.
\end{equation*}
\item We have
\begin{align*}
&[\rho(\alpha_1 \cdot \alpha_2), \lambda(\beta_1 \wedge \beta_2 \wedge \beta_3)] \\
&=\sum_{(i_1,i_2) =(1,2),(2,1),(j_1,j_2,j_3)=(1,2,3),(2,3,1),(3,1,2)}
-2 \mu(\alpha_{i_1},\beta_{j_1}) \lambda(\alpha_{i_2} \wedge \beta_{j_2} \wedge \beta_{j_3})
\mod F^4 \skein{\Sigma}.
\end{align*}
\end{enumerate}
\end{cor}

\begin{cor}
\label{cor_bch_jouken}
Let $V_1$ and $V_2 \subset V_1$ be $\Q$-linear subspaces of $H$
satisfying $\mu (v,v') =0$ for any $v \in V_2 $ and $v' \in V_1$.
We denote
\begin{equation*}
S = \shuugou{x  \in \widehat{\skein{\Sigma}}|
\rho^{-1} (x \mod F^3 \widehat{\skein{\Sigma}}) \in
V_1 \cdot V_2}.
\end{equation*}
Then we have
\begin{equation*}
\sigma (s_1) \circ \sigma(s_2) \circ \cdots \circ \sigma (s_{2i-1}) (F^{i-1}
\skein{\Sigma}) \subset F^i \skein{\Sigma}
\end{equation*}
for any $i \in \Z_{\geq 1}$ and $s_1, \cdots , s_{2i-1} \in S$.

\end{cor}

\begin{proof}
We denote $H_1$ by $V_0$.
Let $V(2i)_j$ be the submodule of $F^{2i} \skein{\Sigma}/F^{2i+1} \skein{\Sigma}$
defined by
\begin{equation*}
V(2i)_j \defeq \sum_{i'=1}^{i}\sum_{j_1+\cdots+j_{2i'} =j} h^{i-i'}\rho (V_{j_1}\cdot
V_{j_2}) \cdots \rho (V_{j_{i'-1}} \cdot V_{j_{i'}}).
\end{equation*}
Let $V(2i+3)_j$ be the submodule of $F^{2i+3} \skein{\Sigma}/F^{2i+4} \skein{\Sigma}$
defined by
\begin{equation*}
V(2i+3)_j \defeq \sum_{i'=1}^{i}\sum_{j_1+\cdots+j_{2i'} =j} h^{i-i'}\rho (V_{j_1}\cdot
V_{j_2}) \cdots \rho (V_{j_{i'-1}} \cdot V_{j_{i'}}) \lambda (V_{j_{i'+1}} \wedge
V_{j_{i'+2}} \wedge V_{j_{i'+3}}).
\end{equation*}
For any $s \in S$, $i$ and $j \in \shuugou{0,1, 2, \cdots, 2i-1}$, 
we have
\begin{align*}
&\sigma (s) (V(i)_j) \subset V(i)_{j+1}, \\
&\sigma (s)(V(i)_{2i}) =0.
\end{align*}
This proves the corollary.
\end{proof}

\subsection{Skein module and $\rho$}
The aim of this subsection is to prove the following.

\begin{prop}
\label{prop_bch_aut_jouken}
Let $V_1$ and $V_2 \subset V_1$ be $\Q$-linear subspaces of $H$
satisfying $\mu (v,v') =0$ for any $v \in V_2 $ and $v' \in V_1$.
We denote
\begin{equation*}
S = \shuugou{x  \in \widehat{\skein{\Sigma}}|
\rho^{-1} (x \mod F^3 \widehat{\skein{\Sigma}}) \in
V_1 \cdot V_2}.
\end{equation*}
Then, for any $i \in \Z_{\geq 1}$ and any finite subset 
$J \subset
\partial \Sigma$, there exists $j_i \in \Z_{\geq 1}$
such that
\begin{equation*}
\sigma (s_1) \circ \sigma(s_2) \circ \cdots \circ \sigma (s_{j_i}) ((\ker \epsilon)^{i-1}
\skein{\Sigma,J}) \subset (\ker \epsilon)^i \skein{\Sigma,J}
\end{equation*}
for any $s_1, \cdots , s_{j_i} \in S$.
\end{prop}

By Leibniz rule and Corollary \ref{cor_bch_jouken}, 
the following lemma induces the above proposition.

\begin{lemm}
\label{lemm_bch_aut_jouken}
We have
\begin{equation*}
\sigma (s_1) \circ \sigma (s_2) \circ \cdots \circ \sigma (s_9) (\skein{\Sigma,J})
\subset \ker \epsilon \skein{\Sigma,J}
\end{equation*}
for any  $J = \shuugou{*_1 , *_2} \subset \partial \Sigma$
and $s_1, \cdots , s_9 \in S$.

\end{lemm}

We consider $\skein{\Sigma,J}/ \ker \epsilon \skein{\Sigma,J}$.
For $r \in \pi (\Sigma,*_1,*_2)
\defeq [(I, \shuugou{0}, \shuugou{1}), (\Sigma,\shuugou{*_1},
\shuugou{*_2})]$, Let $\kukakko{r}$
be the element $\skein{\Sigma,J}/ \ker \epsilon \skein{\Sigma,J}$
of presented by $r$.
By equation (\ref{equation_kouten_sa}),
$\kukakko{\cdot} :\pi(\Sigma,*_1,*_2) \to
\skein{\Sigma,J}/ \ker \epsilon \skein{\Sigma,J}$
is well-defined.
We also denoted $\kukakko{\cdot}:\Q \pi (\Sigma,*_1,*_2)
\to \skein{\Sigma,J}/\ker \epsilon \skein{\Sigma,J}$ 
by its $\Q$-linear extension.
By the skein relation of $\skein{\Sigma,J}$, we have
\begin{equation*}
2\kukakko{yr} =\kukakko{yxr}+\kukakko{yx^{-1} r}
\end{equation*}
for $x,y \in \pi_1 (\Sigma,*_1)$ and 
$r \in \pi (\Sigma,*_1,*_2)$. 
Using this equation, 
we have the following calculations.

\begin{lemm}
\label{lemm_module_calculation}
For $x,y,z \in \pi_1 (\Sigma,*_1)$ and 
$r \in \pi (\Sigma,*_1,*_2)$,
we have
\begin{align*}
&\kukakko{(x-1)(y-1)r}=-\kukakko{(y-1)(x-1)r}, \\
&\kukakko{(x-1)(y-1)(z-1)r}=0.
\end{align*}
\end{lemm}

\begin{proof}
We have
\begin{align*}
&\kukakko{(x-1)(y-1)r} \\
&= \kukakko{-\gyaku{y}\gyaku{x}r+2r-xr-yr+r} \\
&=\kukakko{\gyaku{y}xr-2\gyaku{y}r+2r-xr-yr+r} \\
&=\kukakko{-yxr+2xr+2yr-4r+2r-xr-yr+r} \\
&=-\kukakko{(y-1)(x-1)r}. \\
\end{align*}
Using this equation, we have
\begin{align*}
&\kukakko{(x-1)(y-1)(z-1)r} \\
&=-\kukakko{(y-1)(z-1)(x-1)r} \\
&=\kukakko{(y-1)(x-1)(z-1)r} \\
&=-\kukakko{(x-1)(y-1)(z-1)r}.
\end{align*}
This proves the lemma.
\end{proof}

We consider the action $\sigma_\pi :\Q \pi_\square \times
\Q \pi (\Sigma,*_1,*_2) \to \Q \pi (\Sigma,*_1,*_2)$ defined by
\begin{equation*}
\sigma_\pi(\zettaiti{x}_\square)(r) \defeq 
\sum_{p \in x \cap r} \epsilon(p,x,r)
(r_{*_1 p}x_p r_{p *_2}-r_{*_1 p}(x_p)^{-1} r_{p *_2})
\end{equation*}
for $\zettaiti{x} \in \pi_\square $ and $r \in \pi (\Sigma,*_1, *_2)$ in general position.
For details, see \cite{Kawazumi} Definition 3.2.1.

Since $\sigma( (\ker \epsilon)^2)(\skein{\Sigma,J}) \subset
\ker \epsilon \skein{\Sigma,J}$ and
$ \sigma ( \skein{\Sigma})( \ker \epsilon \skein{\Sigma,J})
\subset \ker \epsilon \skein{\Sigma,J}$,
the action $\sigma :\skein{\Sigma} \times
\skein{\Sigma,J} \to \skein{\Sigma}$ induces
$\sigma : \skein{\Sigma}/ (\ker \epsilon)^2 \times
\skein{\Sigma,J} / \ker \epsilon \skein{\Sigma,J}
\to \skein{\Sigma,J} / \ker \epsilon
\skein{\Sigma,J}$.
By equation (\ref{equation_kouten_sa}),
we have 
\begin{equation*}
\sigma (\kukakko{x})(\kukakko{r})
=-\kukakko{\sigma_\pi (\zettaiti{x}_\square)(r)}.
\end{equation*}
Using this equation, it is enough to show the following
in order to prove Lemma \ref{lemm_bch_aut_jouken}.

\begin{prop}
Let $V_1$ and $V_2 \subset V_1$ be $\Q$-linear subspaces of $H$
satisfying $\mu (v,v') =0$ for any $v \in V_2 $ and $v' \in V_1$.
We denote
\begin{equation*}
S = \shuugou{x  \in \zettaiti{(\ker \epsilon)^2}|
C(2)(x \mod \zettaiti{(\ker \epsilon)^3}) \in
c(V_1 \otimes V_2)}.
\end{equation*}
Then we have the following.
\begin{enumerate}
\item We have $\sigma(s) (\Q \pi_1 (\Sigma,*,*')) \subset
\ker \epsilon_\pi \pi_1 (\Sigma, *,*')$
for any $s \in S$.
\item We have $\sigma (s_1) \circ \sigma (s_2) \circ \sigma
(s_3) (\ker \epsilon_\pi \pi_1 (\Sigma, *,*'))
\subset (\ker \epsilon_\pi)^2 \pi_1 (\Sigma, *,*')$
for any $s_1,s_2,s_3 \in S$.
\item We have  $\sigma (s_1) \circ \sigma (s_2) \circ 
\sigma (s_3) \circ \sigma (s_4) \circ \sigma
(s_5) ((\ker \epsilon_\pi)^2 \pi_1 (\Sigma, *,*'))
\subset (\ker \epsilon_\pi)^3 \pi_1 (\Sigma, *,*')$
for any $s_1,s_2,s_3,s_4,s_5 \in S$.
\end{enumerate}
\end{prop}

\begin{proof}
The first claim is obvious.
Let $V_0$ be $H_1$. We have $
\sigma'_{\pi,2,1} (c(2)(s))(V_i) \subset V_{i+1}$ for $i=0,1$
for any $s \in S$.
Furthermore, we have
$\sigma' (c(2)(s))(V_2) =0$ for any $s in S$.
This proves the second claim.
We define
\begin{equation*}
V(2)_j = \sum_{j_1+j_2=j}c(V_{j_1} \otimes V_{j_2})
\end{equation*}
for $j=0,1,2,3,4$.
We have
$\sigma' _{\pi,2,2} (c(2)(s))(V(2)_i) \subset V(2)_{i+1}$
for $i=0,1,2,3$.
Furthermore, we have 
$\sigma'_{\pi,2,2} (c(2)(s))(V(2)_4) =0$.
This proves the third claim.
\end{proof}

\section{Framed pure braid group}
\label{section_pure_braid_group}
In this section, let $\Sigma$ be a compact connected surface
of genus $0$ and $b+1$ boundary components.
We remark the mapping class group $\mathcal{M} (\Sigma)$ of $\Sigma$
is isomorphic to the famed pure braid group with $b$ strings.

The completed Kauffman bracket skein module $\widehat{\skein{\Sigma}}$
has a filtration $\filtn{F^n \widehat{\skein{\Sigma}}}$
satisfying $\skein{\Sigma}/F^n \skein{\Sigma}\simeq \widehat{
\skein{\Sigma}}/F^n \widehat{\skein{\Sigma}}$ for
$n \in \Z_{\geq 0}$.

\subsection{The Baker-Campbell-Hausdorff series}
By Corollary \ref{cor_bracket_filtration}, we have the following.

\begin{lemm}
\label{lemm_filt_pure}
We have
\begin{align*}
[\widehat{\skein{\Sigma}},F^n \widehat{\skein{\Sigma}}] \subset
F^{n+1} \widehat{\skein{\Sigma}}.
\end{align*}
\end{lemm}

In this paper, we define the Baker-Campbell-Hausdorff series
$\bch$ by
\begin{equation*}
\bch (a_1,a_2, \cdots, a_m) \defeq (-A+\gyaku{A})\log (\prod_{i=1}^m \exp (\frac{a_i}{
-A+\gyaku{A}}))
\end{equation*}
for $a_1, a_2, \cdots, a_{m} \in
\widehat{\skein{\Sigma}}$.
As elements of  the associated Lie algebra $(\widehat{\skein{\Sigma}}, [ \ \ , \ \ ])$,
it has a usual expression.
For example, 
\begin{equation*}
\bch(x,y) = x+y+\frac{1}{2}[x,y]+\frac{1}{12}([x,[x,y]]+[y,[y,x]])+ \cdots.
\end{equation*}
By Lemma \ref{lemm_filt_pure}, the Baker-Campbell-Hausdorff series is well-defined.
The Baker-Campbell-Hausdorff series satisfies
\begin{align*}
& \bch(a,-a) =0, \\
& \bch(0,a)=\bch(a,0) =a, \\
& \bch(a,\bch(b,c))=\bch(\bch(a,b),c), \\
& \bch(a, b, -a) =\exp(\sigma(a))(b).
\end{align*}
Hence $(\widehat{\skein{\Sigma}}, \bch) $ is a group whose identity is $0$.
By Proposition \ref{prop_bch_aut_jouken}, we define
$\exp (\sigma(s))( \cdot) \in \Aut (\widehat{\skein{\Sigma,J}})$
is well-defined for any $s \in \widehat{\skein{\Sigma}}$ and
any finite subset $J \subset \partial \Sigma$.
Furthermore,
$\exp :(\widehat{\skein{\Sigma}},\bch) \to \Aut (\widehat{\skein{\Sigma,J}})$
is a group homomorphism, i.e, $\exp(\sigma(\bch(a,b))) =\exp(\sigma(a)) \circ \exp
(\sigma(b))$ for $a,b \in \widehat{\skein{\Sigma}}$.

The Baker-Campbell-Hausdorff series $\bch$ is defined by
\begin{equation*}
\bch (a_1,a_2, \cdots, a_m) \defeq (-A+\gyaku{A})\log (\prod_{i=1}^m \exp (\frac{a_i}{
-A+\gyaku{A}}))
\end{equation*}
for $a_1, a_2, \cdots, a_{m} \in
\widehat{\skein{\Sigma}}$. For example, 
\begin{equation*}
\bch(x,y) = x+y+\frac{1}{2}[x,y]+\frac{1}{12}([x,[x,y]]+[y,[y,x]])+ \cdots.
\end{equation*}
The Baker-Campbell-Hausdorff series satisfies
\begin{align*}
& \bch(a,-a) =0, \\
& \bch(0,a)=\bch(a,0) =a, \\
& \bch(a,\bch(b,c))=\bch(\bch(a,b),c), \\
& \bch(a, b, -a) =\exp(\sigma(a))(b).
\end{align*}
Hence $(\widehat{\skein{\Sigma}}, \bch) $ is a group whose identity is $0$.
By Proposition \ref{prop_bch_aut_jouken}, 
$\exp (\sigma(s))( \cdot) \in \Aut (\widehat{\skein{\Sigma,J}})$
is well-defined for any $s \in \widehat{\skein{\Sigma}}$ and
any finite subset $J \subset \partial \Sigma$.
Furthermore,
$\exp :(\widehat{\skein{\Sigma}},\bch) \to \Aut (\widehat{\skein{\Sigma,J}})$
is a group homomorphism, i.e, $\exp(\sigma(\bch(a,b))) =\exp(\sigma(a)) \circ \exp
(\sigma(b))$ for $a,b \in \widehat{\skein{\Sigma}}$.

\subsection{The group homomorphism $\zeta :\mathcal{M} (\Sigma) 
\to (\widehat{\skein{\Sigma}}, \bch)$}

The mapping class group $\mathcal{M} (\Sigma)$ is generated by
$\shuugou{t_{ij}|1 \leq i<j \leq b} \cup \shuugou{t_i|1 \leq i \leq b}$,
where $t_{ij} \defeq t_{c_{ij}}$ and $t_i \defeq t_{c_i}$.
Furthermore $\mathcal{M} (\Sigma)$ is presented by the relations
\begin{align*}
&\ad (t_i)(t_j) =t_j, \\
&\ad (t_s)(t_{ij}) =t_{ij}, \\
&\ad(t_{rs})(t_{ij}) =t_{ij} & \mathrm{if\ \ } r<s<i<j, \\
&\ad(t_{rs})(t_{ij}) =t_{ij} & \mathrm{if\ \ } i<r<s<j, \\
&\ad(t_{rs}t_{rj})(t_{ij}) =t_{ij} & \mathrm{if\ \ } r<s=i<j, \\
&\ad(t_{rs}t_{ij}t_{sj})(t_{ij}) =t_{ij} & \mathrm{if\ \ } i=r<s<j, \\
&\ad(t_{rs}t_{rj}t_{sj}\gyaku{t_{rj}} \gyaku{t_{sj}})(t_{ij}) =t_{ij} & \mathrm{if\ \ } r<i<s<j, \\ \end{align*}
where $\ad (a)(b) \defeq a b \gyaku{a}$.
See, for example, \cite{Birman} p.20 Lemma 1.8.2.

\begin{df}
The group homomorphism $\zeta : \mathcal{M} (\Sigma) \to
(\widehat{\skein{\Sigma}},\bch)$ is defined by
$t_i \mapsto L(c_i)$ and $t_{ij} \mapsto L(c_{ij})$,
where $L(c) \defeq \frac{-A+\gyaku{A}}{4 \log (-A)}
(\arccosh (-\frac{c}{2}))^2-(-A+\gyaku{A}) \log (-A)$.
\end{df}

\begin{thm}
\label{thm_zeta_pure_braid}
The group homomorphism $\zeta$ is well-defined and injective.
\end{thm}

\begin{proof}
Let $c, c'_1, c'_2, \cdots ,c'_k$ be elements of $\shuugou{c_i,c_{ij}}$ in $\Sigma$
and $\epsilon_1, \cdots,  \epsilon_k $ be elements of $\shuugou{\pm 1}$
satisfying $t_{c'_1}^{\epsilon_1} t_{c'_2}^{\epsilon_2}
\cdots t_{c'_k}^{\epsilon_k} (c) =c$.
It is enough to check 
\begin{equation*}
\bch(\epsilon_1 L(c'_1), \cdots,
\epsilon_2 L(c'_2), \cdots, \epsilon_k L(c'_k),L(c),
,-\epsilon_k L(c'_k), \cdots, -\epsilon_1 L(c'_1)) =L(c).
\end{equation*}
By Theorem \ref{thm_Dehn_twist}, we have
\begin{align*}
&\bch(\epsilon_1 L(c'_1), \cdots,
\epsilon_2 L(c'_2), \cdots, \epsilon_k L(c'_k),L(c),
,-\epsilon_k L(c'_k), \cdots, -\epsilon_1 L(c'_1)) \\
&=\exp(\sigma(\epsilon_1 L(c'_1))) \circ \cdots \circ \exp(\sigma(\epsilon_k 
L(c'_k)))(L(c)) \\
&=t_{c'_1}^{\epsilon_1} t_{c'_2}^{\epsilon_2}
\cdots t_{c'_k}^{\epsilon_k} (L(c)) =L(c).
\end{align*}
This finishes the proof of well-definedness of $\zeta$.

By definition of $\bch$, we have $\xi (\cdot) = \exp (\sigma(\zeta(\xi)))
(\cdot): \widehat{\skein{\Sigma,J}} \to \widehat{\skein{\Sigma,J}}$
for any $\xi  \in \mathcal{M} (\Sigma)$ and  any finite subset $J$ of 
$\partial \Sigma$.
Using Corollary \ref{cor_map_inj}, we have
$\xi =\id_\Sigma$ if and only if $\zeta(\xi) =0$.
This finishes the proof of injectivity of $\zeta$.
\end{proof}

\begin{rem}
Using the lantern relation 
\begin{equation*}
\label{equaton_lantern}
\bch (L(c_{123}),-L(c_{12}),-L(c_{23}),-L(c_{13}),
L(c_1),L(c_2),L(c_3))=0,
\end{equation*}
we have
\begin{equation*}
\zeta (t_c) =L(c)
\end{equation*}
for any simple closed curve $c$ in $\Sigma_{0,b+1}$.
Details will appear in \cite{TsujiTorelli}.
\end{rem}

\end{document}